\documentclass[12pt]{article}

\usepackage{amsmath,amssymb}
\usepackage{amsthm}
\usepackage{mathtools}
\usepackage{hyperref}
\usepackage{mathrsfs}
\usepackage{MnSymbol,wasysym}
\usepackage{authblk}

\allowdisplaybreaks

{\theoremstyle{plain}%
 \newtheorem{theorem}{Theorem}
 \newtheorem{corollary}[theorem]{Corollary}
 \newtheorem{proposition}[theorem]{Proposition}
 \newtheorem{lemma}[theorem]{Lemma}%
}
{\theoremstyle{remark}
\newtheorem{remark}[theorem]{Remark}
}
{\theoremstyle{definition}

\newtheorem{example}[theorem]{Example}
}

\DeclareMathOperator{\cat}{{\it G}}
\DeclareMathOperator{\barnes}{{\it G_B}}
\DeclareMathOperator{\EMcst}{{\it \gamma}}
\newcommand{\C}{\mathbb{C}}

\newenvironment{dedication}
 {\thispagestyle{empty}
 \vspace*{.5truecm}
 \itshape 
 \raggedleft 
 }
 {\par 
 \vspace{1truecm} 
 }

\title{Gosper's and Ramanujan's infinite products and generalizations}

\author[1]{Jean-Paul Allouche}
\author[2]{John M. Campbell}
\author[3]{Simon R. Holcombe}
\author[4]{Lubomir Markov}
\author[5]{Manon Stipulanti}

\affil[1]{CNRS, IMJ-PRG, Sorbonne, 4 Place Jussieu, 75005 Paris, France\\ \href{mailto:jean-paul.allouche@imj-prg.fr}{\tt jean-paul.allouche@imj-prg.fr}}

\affil[2]{Department of Mathematics and Statistics, Dalhousie University, Halifax, NS B3H 4R2, Canada\\\href{mailto:jh241966@dal.ca}{\tt jh241966@dal.ca}}

\affil[3]{Melbourne Business School
The University of Melbourne
200 Leicester Street, Carlton, VIC, 3053
Australia\\ \href{mailto:s.holcombe@mbs.edu}{\tt s.holcombe@mbs.edu}}

\affil[4]{Department of Mathematics and Computer Science
Barry University
11300 N.E. Second Avenue, Miami Shores, FL 33161
USA\\ \href{mailto:lmarkov@barry.edu}{\tt lmarkov@barry.edu}}

\affil[5]{Department of Mathematics, University of Liège, Allée de la Découverte 12, 4000 Liège, Belgium\\ \href{mailto:m.stipulanti@uliege.be}{\tt m.stipulanti@uliege.be}}

\begin{document}

\maketitle

\begin{dedication}
A tribute to the memory of Jonathan Sondow 1943--2020 
\end{dedication}

\newpage

\begin{abstract}
 Gosper introduced many remarkable formulas involving infinite products. One such formula provides an evaluation of the product $$
 \prod_{n \geq 1} \frac{1}{e}\left(\frac{1}{3n} + 1\right)^{3n + \frac{1}{2}}. $$ It appears that a complete proof of this evaluation has not 
 appeared in the literature. In this paper, we introduce some techniques to extend Gosper's evaluation. We prove a generalization of 
 Gosper's formula involving three free parameters, which can also be applied to evaluate arbitrary multisections of Gosper's product. 
 Then we investigate higher-order extensions involving higher-degree polynomials in $n$, relative to the  polynomials $3n$ and 
 $3n+\frac{1}{2}$ involved in Gosper's product, yielding a generalization of an infinite product evaluation due to Ramanujan and 
recently studied by Bradley--Thrush [\textit{\it Ramanujan J.} (2025)]. We also consider connections with infinite products due to 
 Kurokawa and Rovinski\u{\i}, building on the recent work of the first and fourth authors [\textit{\it Ramanujan J.} (2025)]. 
\end{abstract}

\vspace{0.1in}


\vspace{0.1in}


\section{Introduction}\label{secintro}
 The study of special functions is described by Neretin~\cite{Neretin} as ``speleological'' (actually, this is more precisely, but not only, 
 about the work of Alekseevsky):

\begin{quotation}
\noindent \emph{It seems to have been an investigation in a ``speleological'' style, which is relatively common in theory of 
 special functions.}
\end{quotation}

\noindent Our speleological journey starts with a remarkable identity due to Gosper, which can be found, e.g., on the MathWorld site 
 about infinite products\footnote{Eric W. Weisstein, Infinite Product, From MathWorld--A Wolfram Resource. 
 \url{https://mathworld.wolfram.com/InfiniteProduct.html}.}; it is also mentioned in a paper of the second author 
 \cite{Campbell2024Automatic}. This identity reads 
\begin{equation}\label{Gosper}
\prod_{n \geq 1} \frac{1}{e}\left(\frac{1}{3n} + 1\right)^{3n + \frac{1}{2}} = \sqrt{\frac{\Gamma(\frac{1}{3})}{2\pi}} \ \frac{3^{\frac{13}{24}} 
 \ e^{1+\frac{2\pi^2 - 3\psi\,'(\frac{1}{3})}{12\pi\sqrt{3}}} }{A^4}, 
\end{equation}
 where we let $\Gamma$ be the Gamma function, $\psi$ the digamma function, and $A$ the Glaisher--Kinkelin constant. For 
 background material on special functions, such as $\Gamma$ and $\psi$, see Section~\ref{definitions-lemmas}; see also, e.g., the 
 standard text by Rainville \cite[\S2]{Rainville1960}. We also refer to Section~\ref{definitions-lemmas} for special constants. We did not 
 find a complete proof of \eqref{Gosper} 
in the literature. However, a hint can be found 
online\footnote{\url{https://math.stackexchange.com/questions/4343366}).},
where it is suggested to use \emph{regularization of infinite 
products}.

\medskip

   The  evaluation in  \eqref{Gosper} above will serve as a guiding thread in this paper. We will give several approaches to prove it, as 
 well as generalizations, and point to links with similar 
infinite products. The paper is organized as follows.

\medskip

$*$ Section~\ref{definitions-lemmas} introduces the definitions of special functions and constants that will be used in the paper, as well as useful lemmas.

\medskip

$*$ In Section~\ref{section-gosper}, we give an ``easy'' proof 
of Gosper's Equation~\eqref{Gosper}. 

\medskip

$*$ Section~\ref{section-gosperlike} gives a natural generalization    of Gosper's Equation and several examples of similar infinite products.   
   This generalization involves free parameters, which has the advantage of    allowing us to evaluate arbitrary multisections of the product in   
            \eqref{Gosper}.     It appears that bisections or multisections of this product have 
not previously been considered. Specific choice of parameters gives, 
for example, the natural companions  $$ 
\prod_{n \geq 1} \frac{1}{e}
\left( \frac{1}{3n+1} + 1 \right)^{3n + \frac{3}{2}} = \frac{e^{1 + \frac{3 \psi\,'\left( 
\frac{1}{3} \right) - 2 \pi^2 }{6 \sqrt{3} \pi} }}{2^{3/2} 3^{1/4}} $$
and 
$$ 
\prod_{n \geq 1} 
\frac{1}{e} \left(\frac{1}{4 n+1}+1\right)^{4 n+\frac{3}{2}}
= \frac{e^{1+\frac{\cat}{\pi }} \sqrt{\Gamma \left(\frac{1}{4}\right)}}{2^{11/6} A^{3/2} \sqrt[4]{\pi }} 
$$
to Gosper's formula in Equation~\eqref{Gosper}.

\medskip

$*$ In Section~\ref{sec: Borwein--Dykshoorn function},  we consider how our methods and results relate to  infinite products studied by 
 Borwein and Dykshoorn~\cite{BorweinDykshoorn1993}.

\medskip

$*$ In Section~\ref{sectionSH}, we show how Sondow and Hadjicostas’
results~\cite{SondowHadjicostas2007} can give similar products. Then we deduce Gosper's formula
(see Section~\ref{Gosper-Sondow-Hadjicostas}).

\medskip

$*$ In Section~\ref{sec:higherorder}, we investigate 
higher-order extensions of Gosper's formula given by 
replacing the polynomial expressions of the forms 
$3n$ and $3n+\frac{1}{2}$ appearing in 
Equation~\eqref{Gosper} with polynomials having the same 
higher degree. This also provides a new generalization of 
a Ramanujan product identity recently considered by 
Bradley--Thrush~\cite{BradleyThrush2025}. 
Also we consider links with the theory of multiple sinus 
of Kurokawa and Rovinski\u{\i}, 
building on the recent work of the first and fourth authors 
\cite{AlloucheMarkov2026}.

\medskip

$*$ In Section~\ref{sec:Wallis-hierarchy}, we describe the
{\it Wallis hierarchy}, a family of infinite products studied in
\cite{Holcombe2026}.

\medskip

$*$ In Section~\ref{sec:Dirichlet-approach}, we propose an approach
using Dirichlet series and their truncations to prove Gosper's Equation~\eqref{Gosper}.

\section{Definitions and lemmas}\label{definitions-lemmas}

In this section, we give definitions and properties that will be 
useful later in the text. Even if many of them are well-known, we
put them here for further reference and self-containment. 

\medskip

   The \emph{Euler--Mascheroni constant} $\EMcst$ is defined by
$$
\EMcst := \lim_{n \to +\infty} 
\left(\sum_{1 \leq k \leq n} \frac{1}{k} - \log n\right) 
= \sum_{k \geq 1} \left(\frac{1}{k} 
- \log\left(1 + \frac{1}{k}\right)\right).
$$

\medskip

  If the real part of the complex number $x$ is positive,   
  then the \emph{Gamma function $\Gamma$} at $x$ is   defined as   
  the (absolutely convergent) integral
$$
\Gamma(x) := \int _{0}^{+\infty} t^{x-1}e^{-t}\,dt.
$$
Clearly, we have $\Gamma(1)=1$ and
\begin{equation}
\label{eq:Gamma+1}
 \Gamma(x+1)=x\Gamma(x).
\end{equation}
Thus, by induction, we obtain
\begin{equation}
\label{eq:Gamma natural}
\Gamma(n)=(n-1)! 
\end{equation}
for all integers $n\ge 1$.
One of the most famous values of the Gamma function is $\Gamma(\frac{1}{2})=\sqrt{\pi}$.
We also recall the reflection formula\footnote{The ``reflection formula''
is called ``La formule des compl\'ements'' in French, and ``Der
Ergänzungssatz'' or ``Die Reflexionsformel'' in German.}
\begin{equation}
\label{eq:formula for prod of gammas}
\Gamma(x) \cdot \Gamma(1-x) = \frac{\pi}{\sin(\pi x)}
\end{equation}
for $x \notin {\mathbb Z}$
(see, e.g., \cite[pp.~14--15 and p.~149]{R-Campbell1966}).

\medskip

For any complex number $s$ with $\Re(s) > 1$, we let $\zeta(s) := \sum_{n \geq 1} \frac{1}{n^s}$ be the \emph{Riemann zeta function}.
For a complex number $s$ with $\Re(s) > 1$ and $a \notin \{0, -1, -2, \ldots\}$, we let 
$\zeta(s,a) := \sum_{n \geq 0} \frac{1}{(n+a)^s}$ be the
\emph{Hurwitz zeta function}, and we also define $\zeta'(-1,x) := 
\frac{\partial}{\partial s} \zeta(s,x)_{|s =-1}$.

\medskip

The \emph{Kinkelin--Bendersky hyperfactorial function $K$}~\cite{Bendersky1933,Kinkelin1860} is the real-valued function $K(x)$ 
defined for $x\ge 0$ by the relation
\[
 \log K(x) := \frac{x^2-x}{2} - \frac{x}{2} \log 2\pi + \int_{0}^{x} \log \Gamma(y) \,dy, 
\]
see \cite[Def.~3]{SondowHadjicostas2007}.
It satisfies $K(0)=K(1)=K(2)=1$ and, for all $x>0$,
\begin{equation}
\label{eq:K+1}
K(x+1)=x^x K(x),
\end{equation}
see~\cite[p.~279]{Bendersky1933} for example.
Consequently, by induction, we have $K(n+1)=1^1 2^2 \cdots n^n$
for any integer $n\ge 0$. (Note that this is where $K$ takes its 
attribute of being ``hyperfactorial''.) The first few values of $K$ 
at integers are stored under~\cite[A002109]{Sloane}. Another 
expression of $K(x)$ is given by
\[
K(x) = e^{\zeta'(-1,x)-\zeta'(-1)},
\]
see \cite[Theorem~1]{kuro-ochiai} for example.

\medskip

The \emph{Glaisher--Kinkelin constant $A$} traces back to $19^{\text{th}}$-century works by Glaisher \cite{Glaisher1893,Glaisher1894,Glaisher1878} and Kinkelin \cite{Kinkelin1860}. 
It is defined by
\[
A := \lim_{x\to+\infty} 
\frac{K(x+1)}{ x^{\frac{x^2+x}{x}+\frac{1}{12}} e^{-\frac{x^2}{4}}}
= \lim_{n\to+\infty} \frac{1^1 2^2 \cdots n^n}{n^{\frac{n^2+n}{n}+\frac{1}{12}} e^{-\frac{n^2}{4}}},
\]
see~\cite[pp. 263–264]{Bendersky1933}  and~\cite[p.~135]{SondowHadjicostas2007}.

\bigskip

We have the following relations between the objects we have 
just introduced. For the first identity below, see, 
e.g., \cite[Eq.~22, p.~558]{SrivastavaChoi2012} and for the second, 
see, e.g., \cite[p. 281]{Bendersky1933}.

\begin{lemma}
\label{lem:relation-K-A}
The Euler--Mascheroni and Glaisher--Kinkelin constants satisfy
\[
\log A = \frac{1}{12} - \zeta'(-1) = \frac{\EMcst + \log 2\pi}{12} - \frac{\zeta'(2)}{2\pi^2}.
\]
The Kinkelin--Bendersky function $K$ and the Glaisher--Kinkelin constant $A$ are related by the following equality: for any integer $n\ge 2$, we have
\[
\prod_{j=1}^{n-1} K\left(\frac{j}{n}\right) = 
\frac{A^{\frac{n^2-1}{n}}}{n^{\frac{1}{12n}}e^{\frac{n^2-1}{12n}}}.
\]
\end{lemma}

\bigskip

We let $\cat$ denote the \emph{Catalan constant} defined by
\begin{equation}\label{def:Catalan-constant}
 \cat := \sum _{n\ge0} \frac{(-1)^{n}}{(2n+1)^{2}}.
\end{equation}

\bigskip

In relation with the Gamma function, we consider successive derivatives of its logarithm.
For an integer $n\ge 0$, we let $\psi^{(n)}(x)$ denote \emph{the polygamma function} defined by 
\[
\psi^{(n)}(x) := 
\frac{d^{n+1}}{dx^{n+1}} \log \Gamma(x).
\]
Note the relation $\psi^{(n)}(x) = (-1)^{n+1} \, n! \, \zeta(n+1,x)$.
The notation $\psi^{(n)}$ can be extended to negative integers $n<0$ as
follows. It is standard to let 
\[
\psi^{(-1)}(x):=
\begin{cases}
\log \Gamma(x) & \text{if $x > 0$}, \\
\overline{\log \Gamma(x)} & \text{if $x < 0$},
\end{cases}
\]
where we let $\overline{z}$ denote the complex conjugate of 
$z \in \mathbb{C}$.
With this convention, for $n \geq 2$, we may define 
$\psi^{(-n)}(x)$ so that $\psi^{(-n)}(0) := 0$
and $\frac{d}{dx} \psi^{(-n)}(x) := \psi^{(-n+1)}(x)$. In
particular, note that, for $x \geq 0$, 
\begin{equation}
\label{eq:psineg2} 
\psi^{(-2)}(x) = \int_{0}^{x} \log \Gamma(y) \, dy. 
\end{equation}
We consider two special cases of polygamma functions: the \emph{digamma} and \emph{trigamma functions} $\psi$ and $\psi\,'$ are respectively defined by
\begin{equation}\label{eq:digamma-trigamma}
\psi(x) := \frac{d}{dx} \log \Gamma(x) \ \ \text{and} \ \ 
\psi\,'(x) := \frac{d^2}{dx^2} \log \Gamma(x) 
= \sum_{n \ge 0} \frac{1}{(n+x)^2}.
\end{equation}
Note that the trigamma function satisfies $\psi\,'(x+1) = 
\psi\,'(x)-{\frac {1}{x^{2}}}$.
Finally, let us recall the 
classical formula
\begin{equation}
\label{eq:formula for psi}
\psi(1+a) = - \EMcst + \sum_{n \geq 1} \frac{a}{n(n+a)} 
\end{equation}
where $a \notin \{-1, -2, -3, \ldots\}$; see, e.g., \cite[6.3.16, p~259]{AbramowitzStegun}.

\bigskip

The {\em Barnes-$G$ function $\barnes$} is defined by
\begin{equation}\label{eq:GBarnes}
\barnes(z+1) = \frac{(2\pi)^{\frac{z}{2}}}{e^{\frac{1}{2}(z(z+1) 
+ \EMcst z^2)}}\prod_{n\ge 1} 
\frac{(1 + \frac{z}{n})^n}{e^{z - \frac{1}{2n}z^2}} 
\end{equation}
for $z\in {\mathbb C}$; see, e.g., \cite[p.~2]{Adamchik2001}.
Note that 
\begin{align}
\label{eq:GB+1}
 \barnes(z+1) = \Gamma(z)\barnes(z).
\end{align}
The values of $\barnes(x)$ for $x \in \{\frac{1}{6}, \frac{1}{4}, \frac{1}{3}, \frac{1}{2}, \frac{2}{3}, \frac{5}{6}\}$ are stored 
in~Table~\ref{tab:GB-special-values}; see, e.g., \cite{Adamchik2001}. 

\begin{table}[ht]
 \centering
 \setlength{\tabcolsep}{10pt} 
 \renewcommand{\arraystretch}{1.5} 
 \[
 \begin{array}{c|c}
 x & \log \barnes(x) \\
 \hline
 \frac{1}{6} & - \frac{\log 12}{144} + \frac{\pi}{20\sqrt{3}} - \frac{5}{6} \log \Gamma\left(\frac{1}{6}\right) - \frac{5}{6} \log A - \frac{\psi\,'(\frac{1}{6})}{40\pi \sqrt{3}} + \frac{5}{72}\\ 
 \frac{1}{4} & - \frac{\cat}{4\pi} - \frac{3}{4} \log \Gamma\left(\frac{1}{4}\right) - \frac{9}{8} \log A +\frac{3}{32}\\
 \frac{1}{3} & \frac{\log 3}{72} + \frac{\pi}{18\sqrt{3}} - \frac{2}{3} \log \Gamma\left(\frac{1}{3}\right) - \frac{4}{3} \log A - \frac{\psi\,'(\frac{1}{3})}{12\pi \sqrt{3}} + \frac{1}{9}\\
 \frac{1}{2} & \frac{\log 2}{24} - \frac{\log \pi}{4} - \frac{3 \log A}{2} + \frac{1}{8}\\
 \frac{2}{3} & \frac{\log 3}{72} + \frac{\pi}{18\sqrt{3}} - 
 \frac{1}{3} \log \Gamma\left(\frac{2}{3}\right) - \frac{4}{3} \log A - \frac{\psi\,'(\frac{2}{3})}{12\pi \sqrt{3}} + \frac{1}{9}\\
 \frac{5}{6} & - \frac{\log 12}{144} + \frac{\pi}{20\sqrt{3}} - \frac{1}{6} \log \Gamma\left(\frac{5}{6}\right) - \frac{5}{6} \log A - \frac{\psi\,'(\frac{5}{6})}{40\pi \sqrt{3}} + \frac{5}{72}
 \end{array}
 \]
 \caption{The value of $\log \barnes(x)$ for $x\in \{\frac{1}{6}, \frac{1}{4}, \frac{1}{3}, \frac{1}{2}, \frac{2}{3}, \frac{5}{6}\}$. 
 }
 \label{tab:GB-special-values}
\end{table}

The functions $K$, $\Gamma$, and $\barnes$ are related as follows.
Using the relation
\[
\int_0^z \log\Gamma(x+1)\, dx = \frac{z}{2} \log 2\pi 
- \frac{z(z+1)}{2} + z \log\Gamma(z+1) - \log \barnes(z+1)
\]
(see~\cite[Section~6.441]{GradRyzh}) and the definition of $K$, one has
\begin{align}
\label{eq:relation-K-Gamma-GB}
 K(x) = \frac{\Gamma(x)^x}{\barnes(x+1)} 
 = \frac{\Gamma(x)^{x-1}}{G_B(x)};
\end{align}
also see~\cite[p.~312]{SondowHadjicostas2007}.
This yields the values of $K(x)$ 
in terms of $\Gamma(x)$ for 
$x \in E = \{\frac{1}{6}, \frac{1}{4}, \frac{1}{3}, \frac{1}{2}, \frac{3}{4}, \frac{5}{6}\}$, as in 
 Table~\ref{tab:K-special-values}. 
 Observe that the formula for $K(\frac{1}{2})$ follows
 in a direct way from Lemma~\ref{lem:relation-K-A}.
 Similarly, the 
 evaluation for $K(\frac{1}{4})$
 may be obtained from Lemma~\ref{lem:relation-K-A} in conjunction with 
 Gosper's formula \cite[p.~71]{Gosper1997} 
\begin{equation}\label{eq:gosper-3}
\int_{0}^{\frac{1}{4}} \log t! \,dt 
= \frac{3\EMcst + 7 \log \pi - 9\log 2}{32} - \frac{9\zeta'(2)}{16\pi^2} + \frac{\cat}{4\pi} - \frac{1}{4}. 
\end{equation}
 Similarly, the evaluation 
 of $K(\frac{1}{3})$ may be obtained
 using Equation~\eqref{eq:relation-K-Gamma-GB} 
 and known values for $\barnes(\frac{1}{3})$ and 
$\barnes(\frac{2}{3})$ given in Table~\ref{tab:GB-special-values}. 
 
\begin{table}[ht]
 \centering
 \setlength{\tabcolsep}{10pt} 
 \renewcommand{\arraystretch}{1.5} 
 \[
 \begin{array}{c|cccc}
 x & \frac{1}{4} & \frac{1}{3} & \frac{1}{2} & \frac{2}{3} \\
 \hline
 K(x) & A^{\frac{9}{8}} e^{\frac{\cat}{4\pi}
 - \frac{3}{32}} & \frac{A^{\frac{4}{3}} e^{\frac{\psi\,'(\frac{1}{3})}{12\pi \sqrt{3}}-\frac{\pi}{18\sqrt{3}} - \frac{1}{9}}}{3^{\frac{1}{72}}} & \frac{A^{\frac{3}{2}}}{2^{\frac{1}{24}} e^{\frac{1}{8}}} & \frac{A^{\frac{4}{3}} e^{\frac{\psi\,'(\frac{2}{3})}{12\pi \sqrt{3}}-\frac{\pi}{18\sqrt{3}} - \frac{1}{9}}}{3^{\frac{1}{72}}}
\end{array}
 \]
 \caption{The value of $K(x)$ for $x\in \{\frac{1}{4}, \frac{1}{3}, \frac{1}{2}, \frac{2}{3}\}$. 
 }
 \label{tab:K-special-values}
\end{table}

\begin{remark}
In the proof of the previous lemma, the formula taken from \cite[p.~71]{Gosper1997} gives Equation~\eqref{eq:gosper-3} where $\cat$ is replaced with 
$\lambda = \frac{1}{8} (\psi\,'(\frac{1}{4}) - \pi^2)$.
Since $\psi\,'$ is the trigamma function defined above, 
$\lambda$ and $\cat$ are indeed the same constant.
\end{remark}

We will also use the following classical result 
(see, e.g., \cite[Section~12.13]{WW} or Theorem 1.1 and its proof
in \cite{ChamberlandStraub2013}).

\begin{lemma}\label{lem:gammapro}
Let $d$ be a positive integer and $a_j,b_j$ complex numbers for $j\in\{1,\ldots,d\}$.
If $\sum_{j=1}^d a_j = \sum_{j=1}^d b_j$ and no $b_j$'s belong to $\{-1, -2, -3, \ldots\}$, then
\[
\prod_{n \geq 1}\frac{(n+a_1) \cdots (n+a_d)}{(n+b_1) \cdots (n+b_d)}
= \frac{\Gamma(1+b_1) \cdots \Gamma(1+b_d)}{\Gamma(1+a_1) \cdots \Gamma(1+a_d)}.
\] 
\end{lemma}

\begin{remark}
One can prove the following generalization of
Lemma~\ref{lem:gammapro}. Let $r$ be a positive integer. 
Let $a_1, \ldots, a_r, \theta_1, \ldots \theta_r$ be $2r$ 
complex numbers such that: 
\begin{itemize} 
\item[]{$*$} no $a_j$ belongs to 
$\{0, -1, -2, -3, \ldots\}$, 
\item[ ]{$*$} $\sum_{1 \leq j \leq r} \theta_j = 0$,
\item[ ]{$*$} $\sum_{1 \leq j \leq r} a_j \theta_j = 0$. 
\end{itemize}
Then, building on classically known techniques for the expression of 
products of rational functions with the $\Gamma$-function (see, e.g., 
Chamberland and Straub's applications of such techniques 
\cite{ChamberlandStraub2013}), it can be shown that 
\begin{align*}
\prod_{n \geq 0} 
\prod_{1 \leq j \leq r}(n + a_j)^{\theta_j}
= \left(\prod_{1 \leq j \leq r} \Gamma(a_j)^{\theta_j}\right)^{-1}.
\end{align*} 
\end{remark}

\section{A quick proof of Gosper's equation}\label{section-gosper}

Let $P=\prod_{n \geq 1} \frac{1}{e}\left(\frac{1}{3n} + 1\right)^{3n + \frac{1}{2}}$ denote Gosper's product in Equation~\eqref{Gosper}. Thus
$$
P^2 = \prod_{n \geq 1} 
\frac{1}{e^2}\left(\frac{1}{3n} + 1\right)^{6n + 1} 
= \displaystyle \left(\prod_{n \geq 1} 
\left(\frac{1}{3n} + 1\right) e^{- \frac{1}{3n}}\right) \
\left(\prod_{n \geq 1}
\left(\frac{1}{3n} + 1\right)^{6n} e^{- 2 + \frac{1}{3n}}\right).
$$
The Weierstrass product for (the inverse of) the Gamma function
(see \cite{JensenGronwall1916}, also see \cite{R-Campbell1966})
reads: for all complex numbers $z$,
$$
 \frac{1}{z e^{\EMcst z} \Gamma(z)} = \prod_{n \geq 1}\left(1+\frac{z}{n}\right)e^{-\frac{z}{n}}.
$$
   Also,   Equation~\eqref{eq:GBarnes}      can be rewritten as    $$ 
\frac{\barnes(1+z)}{(2\pi)^{\frac{z}{2}} 
\ e^{-\frac{1}{2}(z(z+1)+\EMcst z^2)}} = \prod_{n\geq 1}\left(1+\frac{z}{n}\right)^n 
e^{-z+\frac{z^2}{2n}}.  $$
Thus, multiplying these two equalities, we have
\begin{equation}\label{generalizewithBarnes}
\frac{G_{B}^6(1+z)}{(2 \pi)^{3z} e^{-3(z(z+1)+\EMcst z^2)}
z e^{\EMcst z} \Gamma(z)} = \
\prod_{n \geq 1}\left(1 + \frac{z}{n}\right)^{6n+1}
e^{- 6z + \frac{3 z^2}{n} - \frac{z}{n}}.
\end{equation}
Setting $z = \frac{1}{3}$ in Equation~\eqref{generalizewithBarnes}, 
this leads to 
$$
 \frac{3 e^{4/3} G_{B}^{6}\left(\frac{4}{3}\right) }{2 \pi \Gamma \left(\frac{1}{3}\right)} 
= P^2.
$$ 
   Now we use Relation~\eqref{eq:GB+1} and the value of   $\log \barnes(1/3)$ from Table~\ref{tab:GB-special-values} 
to obtain the value of $P^2$, and so finally Equation~\eqref{Gosper}.

\medskip

\begin{remark}
It is interesting to note that     {\tt Mathematica 13}   does not give the value of $P$, nor {\it a fortiori} 
Equation~\eqref{generalizewithBarnes}.
\end{remark}

\section{A four-parameter generalization}\label{section-gosperlike}
Remarkable results due to Gosper \cite{Gosper1997} on the evaluation 
of integrals related to Equation~\eqref{eq:psineg2} inspire the 
development of new techniques linking $\psi^{(-2)}$ 
and extensions of Gosper's product in~\eqref{Gosper}. 
 We consider infinite 
products of the type
\begin{equation}
\label{eq:Gosker-like-product-alpha}
\prod_{n \geq 1} \frac{1}{e} 
\left(\frac{1}{\alpha n} + 1 \right)^{\alpha n + \frac{1}{2}}, 
\end{equation}
where $\alpha$ is a positive real, or, more generally,
$$
\prod_{n \geq 1} e^{f(n,x)} 
\left(\frac{x}{\mu n + \nu} + 1 \right)^{\alpha n + \beta},
$$
where $\alpha, x, \mu$ are positive reals, $\beta, \nu$ are nonnegative
reals, and the function $f$ is chosen ``economically'' so that
the product converges. Actually, we will limit ourselves to
the case for such an $f$ when $\mu = \alpha$, $\nu \in\{0, 1\}$.
More precisely, in the next sections, we will evaluate the products
\begin{equation}
\label{eq:B-alpha-beta}
B_{\alpha, \beta}(x) :=
\prod_{n \geq 1} e^{-x - \frac{\beta x}{\alpha n}
+ \frac{x^2}{2 \alpha n}}
\left( \frac{x}{\alpha n} + 1 \right)^{\alpha n + \beta} 
\end{equation}
and
\begin{equation}
\label{eq:C-alpha-beta}
C_{\alpha, \beta}(x) := \prod_{n \geq 1} 
e^{-x + \frac{1-\beta}{\alpha n} x + \frac{x^2}{2 \alpha n}} 
\left( \frac{x}{\alpha n + 1} + 1 \right)^{\alpha n +\beta} 
\end{equation}
for positive reals $\alpha, x$ and a nonnegative real $\beta$.

\subsection{A first round}\label{natural}

First, motivated by Gosper's product, we study products of the form given in Equation~\eqref{eq:Gosker-like-product-alpha} 
or even more generally, the product
$$
\prod_{n \geq 1} e^{f(n)} 
\left(\frac{1}{\alpha n} + 1 \right)^{\alpha n + \beta}.
$$
where $f$ will be determined.
To that end we consider the function $B_{\alpha, \beta}(x)$ of Equation~\eqref{eq:B-alpha-beta} for positive reals $\alpha, \beta, x$.

\begin{theorem}\label{B-alpha-beta}
For the function $B_{\alpha, \beta}$ defined in Equation~\eqref{eq:B-alpha-beta}, we have
\begin{multline}\label{variant1}
\log B_{\alpha, \beta}(x) = \displaystyle
(x - \beta) \left(\log \Gamma\left(\frac{x}{\alpha} + 1 \right) 
+ \EMcst\frac{x}{\alpha}\right) \\
- \ \alpha \int_0^{x/\alpha} \log \Gamma(t+1) \ dt
- \frac{\EMcst}{2\alpha} x^2 \ \ \ \ \ \ \ \ \ \ 
\end{multline}
and
\begin{equation}\label{variant2}
B_{\alpha, \beta}(x) = 
\frac{G_{B}^{\alpha}\left(\frac{x}{\alpha}+1\right)}
{\Gamma^{\beta} \left(\frac{x}{\alpha}+1\right)}
e^{\frac{1+\EMcst}{2\alpha} x^2 - \left(\frac{\beta \EMcst}{\alpha}
+ \frac{\log2\pi -1}{2}\right) x},
\end{equation}
for positive reals $\alpha, \beta, x$.
\end{theorem}
\begin{proof}
Recalling the definition of $B_{\alpha, \beta}$ in
Equation~\eqref{eq:B-alpha-beta}, we have
\begin{equation}\label{logBalphabeta}
\log B_{\alpha, \beta}(x) =
\sum_{n \geq 1} \left( - x - \frac{\beta x}{\alpha n}
+ \frac{x^2}{2 \alpha n} 
+ (\alpha n + \beta) \log \left(\frac{x}{\alpha n} + 1 \right)\right),
\end{equation}
\bigskip
 Differentiating both sides of Equation~\eqref{logBalphabeta}, 
 and then interchanging the application of differentiation and infinite summation on the resulting right-hand side via the dominated convergence theorem,
 we find that
\begin{equation}\label{expresswdigamma} 
\frac{B_{\alpha, \beta}'(x)}{B_{\alpha, \beta}(x)} 
=
 \frac{x - \beta}{\alpha} 
 \sum_{n \geq 1} \frac{x/\alpha}{n (n + x/\alpha)}.
\end{equation}
 Using Formulas~\eqref{eq:formula for psi} 
 and \eqref{expresswdigamma} together yields 
\begin{equation}\label{BprimeoverB}
\frac{B_{\alpha, \beta}'(x)}{B_{\alpha, \beta}(x)} =
\left(\frac{x-\beta}{\alpha}\right)
\left(\psi\left(\frac{x}{\alpha}+1\right) + \EMcst\right).
\end{equation}
 Using the definition for the digamma function given in Equation~\eqref{eq:digamma-trigamma} together with 
 the evaluation $B_{\alpha, \beta}(0)=1$, we obtain
 Equation~\eqref{variant1} by integrating 
Equation~\eqref{BprimeoverB}. 

 To obtain the second relation of the theorem, we make use of the relation
\begin{equation}\label{alexeiewsky-barnes}
\int_0^z \log \Gamma(1+t) \ dt = \frac{1}{2}(\log 2\pi - 1)z
- \frac{z^2}{2} + z \log \Gamma(1+z) - \log \barnes(1+z) 
\end{equation}
 attributed to 
 Alexeiewsky
 (see~\cite[p.\ 281]{Barnes1900}
 and Remark~\ref{rem:attribute} below). 
 In this direction, Equations~\eqref{alexeiewsky-barnes} and~\eqref{variant1} together yield
\begin{multline*}
\log B_{\alpha, \beta}(x) = 
- \beta \log\Gamma\left(\frac{x}{\alpha}+1\right)
+ \frac{1+\EMcst}{2\alpha} x^2 \\
- \left(\frac{\beta \EMcst}{\alpha} 
+ \frac{\log 2\pi - 1}{2}\right) x
+ \alpha \displaystyle \log \barnes \left(\frac{x}{\alpha} + 1\right),
\end{multline*}
which gives Equation~\eqref{variant2} in the statement of the theorem. 
%
\end{proof}

\begin{remark}\label{rem:attribute}
Barnes \cite[p.\ 281]{Barnes1900} attributes 
Relation~\eqref{alexeiewsky-barnes} to Alexeiewsky 
(see the comments about the papers of Alexeiewsky 
in \cite[p.\ 265]{Barnes1900}). Note that the expression 
given by Barnes involves the term ``$z.z+1$'' that should 
be interpreted as ``$z(z+1)$'', as it becomes clear by 
taking $a=1$ in the formula given 
in \cite[p.\ 282]{Barnes1900}. For more about the work 
of Alexeiewsky, we refer the reader to \cite{Neretin}. 
\end{remark}

   By choosing particular values for $x$, $\alpha$, $\beta$ in    Equations~\eqref{variant1} and~\eqref{variant2} of  
 Theorem~\ref{B-alpha-beta},   we can obtain several relations that are interesting and relatively 
simple---including Gosper's Equation~\eqref{Gosper}. Here are some 
examples.

\begin{example}
\label{ex:ex of products}
Using Theorem~\ref{B-alpha-beta} above, and either \cite{Gosper1997} 
for relevant values of the integral of $\log\Gamma(t+1)$, or
\cite{Adamchik2001} for relevant values of $\barnes(1+z)$, we obtain the 
following identities (where we recall that $\cat$ is the Catalan constant in Equation~\eqref{def:Catalan-constant}):
\begin{equation}\label{eq:ex-of-products-1}
\prod_{n \geq 1} \frac{1}{e}\left(\frac{1}{3n} + 1\right)^{3n + \frac{1}{2}}
= \ \sqrt{\frac{\Gamma(\frac{1}{3})}{2\pi}} 
\ \frac{3^{\frac{13}{24}} \
e^{1 + \frac{2\pi^2 - 3\psi\,'(\frac{1}{3})}{12\pi\sqrt{3}}} }{A^4},
\end{equation}
\begin{equation*}\label{eq:ex-of-products-2}
 \prod_{n \geq 1} \frac{1}{e}\left(\frac{1}{2n} + 1\right)^{2n + \frac{1}{2}}
= \ \frac{2^{\frac{1}{12}} \ e}{\pi^{\frac{1}{4}} \ A^3},
\end{equation*}
\begin{equation*}\label{eq:ex-of-products-3}
 \prod_{n \geq 1} \frac{1}{e^{1 - \frac{1}{n}}}\left(\frac{1}{2n} + 1\right)^{2n - \frac{3}{2}}
= \ \frac{\pi^{\frac{3}{4}} \ e^{1+\EMcst}}{2^{\frac{23}{12}} \ A^3},
\end{equation*}
\begin{equation}\label{eq:ex-of-products-4}
 \prod_{n\ge 1} e^{\frac{1-2\beta}{8n} - 1}\ \left(\frac{1}{4n}+1\right)^{4n+\beta} 
= \ \frac{2^{2\beta - \frac{1}{2}}\ \Gamma\left(\frac{1}{4}\right)^{1-\beta}}
{\sqrt{\pi} \ A^{\frac{9}{2}}} \
 \ e^{1 - \frac{\cat}{\pi} - \frac{\EMcst (2\beta-1)}{8}}, 
\end{equation}
\begin{equation*}\label{eq:ex-of-products-5}
 \prod_{n \geq 1} \frac{1}{e} 
\left(\frac{1}{6n} + 1 \right)^{6n+\frac{1}{2}} 
= \ \sqrt{\frac{\Gamma\left(\frac{1}{6}\right)}{\pi }} \
\frac{2^{- \frac{1}{12}} 3^{\frac{11}{24}} 
e^{1+\frac{\pi }{2 \sqrt{3}}-\frac{\sqrt{3} 
\psi\,'\left(\frac{1}{3}\right)}{4 \pi }}}{A^5}.
\end{equation*}
%
%
Observe that Equation~\eqref{eq:ex-of-products-1} is Gosper's product.
In particular, Relation~\eqref{eq:ex-of-products-4} yields, e.g., 
\begin{align*}
\prod_{n \geq 1} \frac{1}{e} 
\left(\frac{1}{4n} + 1\right)^{4n + \frac{1}{2}}
&
= \ \sqrt{\frac{2 \ \Gamma\left(\frac{1}{4}\right)}{\pi \ A^9}} \ \
e^{1 - \frac{\cat}{\pi}}, 
\\
\prod_{n\ge 1} e^{\frac{1}{n} - 1} \left(\frac{1}{4n} + 1\right)^{4n-\frac{7}{2}} 
&
= \ \sqrt{\frac{\Gamma\left(\frac{1}{4}\right)^9}{2^{15}\ \pi \ A^9}} \ \
e^{1+ \EMcst - \frac{\cat}{\pi}},
\\
\prod_{n\ge 1} e^{\frac{1}{8n}-1} \left(\frac{1}{4n}+1\right)^{4n} 
&
= \ \frac{\Gamma\left(\frac{1}{4}\right)}{\sqrt{2\pi A^9}} \ \ 
e^{1 - \frac{\cat}{\pi} + \frac{\EMcst}{8}}. 
\end{align*}
\end{example} 

%

%

\begin{remark}
\label{eq: special value for 1, a half and 1}
The case $\alpha = 1$, $\beta = 1/2$, $x = 1$ of 
Theorem~\ref{B-alpha-beta} above might deserve a special mention.
It reads
$$
\prod_{n \geq 1} \frac{1}{e}
\left(\frac{1}{n}+1\right)^{n + \frac{1}{2}}
= \frac{e}{\sqrt{2 \pi}}.
$$
This equality can also be proven by a more direct argument. 
Namely, write 
\begin{align*}
\prod_{n \geq 1} \frac{1}{e}
\left(\frac{1}{n}+1\right)^{n + \frac{1}{2}}
&= \displaystyle \lim_{N \to + \infty}
\prod_{1 \leq n \leq N} \frac{1}{e}
\left(\frac{1}{n}+1\right)^{n + \frac{1}{2}} \\
&= 
\displaystyle \lim_{N \to + \infty} \prod_{1 \leq n \leq N}
\frac{1}{e} \frac{(n+1)^{n+1}}{n^n} \frac{1}{\sqrt{n(n+1)}} \\
&= \displaystyle \lim_{N \to + \infty}
\frac{(N+1)^{N + \frac{1}{2}}}{e^N\,N!}, 
\end{align*}
where the last equality is obtained by telescopic simplifications.
Now, the Stirling approximation of $N!$ yields that the last limit is $e/\sqrt{2\pi}$, as desired.
\end{remark}

\medskip

\begin{remark}
If we take $\alpha = 1$ and $\beta = 0$ in Theorem~\ref{B-alpha-beta},
we obtain
$$
B_{1, 0}(x) =
\prod_{n \geq 1} e^{-x + \frac{x^2}{2n}}
\left( \frac{x}{n} + 1 \right)^n =
\barnes\left(x+1\right)
e^{\frac{1+\EMcst}{2} x^2 - \left(\frac{\log2\pi -1}{2}\right) x},
$$
which is a property of (or one way of defining) the Barnes-$G$ function 
(see \cite[p.~268]{Barnes1900}, or, e.g., \cite[p.~136]{Finch2003}).
\end{remark}

\subsection{Multisections}
  The generalization in Section~\ref{natural} allows us to evaluate    some multisections of the product in Equation~\eqref{Gosper}: 
it appears that bisections or multisections of this product have 
not been considered in the literature. We will restrict ourselves 
to writing down a particular case, leaving the pleasure 
of discovering more general results to the reader.


%
%
%
%
%

\begin{theorem}\label{C-alpha-beta}
Let $C_{\alpha,\beta}(x)$ be the function defined in Equation~\eqref{eq:C-alpha-beta}. Then
\begin{multline*}
C_{\alpha,\beta}(x) =
 \frac{\Gamma^{\beta -1} \left(\frac{\alpha +1}{\alpha }\right)}
 {\Gamma^{ \beta -x-1} \left(\frac{x+\alpha +1}{\alpha }\right)} 
 \ \cdot \\
 \exp\left(
 \alpha \psi ^{(-2)}\left(1+\frac{1}{\alpha }\right)+\frac{\EMcst x (-2 \beta +x+2)}{2 \alpha }-\alpha \psi ^{(-2)}\left(\frac{x+\alpha
 +1}{\alpha }\right) \right)
\end{multline*}
for all positive reals $x, \alpha$, and
nonnegative real $\beta$.
\end{theorem}

The proof of Theorem~\ref{C-alpha-beta} closely mimics that of Theorem~\ref{B-alpha-beta}.

\begin{remark}
Note that the equality in Theorem~\ref{C-alpha-beta}
above can be expressed in a form similar to the statement
of Theorem~\ref{B-alpha-beta} by using 
Equation~\eqref{alexeiewsky-barnes}.
\end{remark}

\begin{example}\label{multi}
By applying Theorem~\ref{C-alpha-beta} we can obtain the following 
equalities:
\begin{equation*}\label{variant2n+1}
\prod_{n \geq 1} \frac{1}{e}
\left( \frac{1}{2n+1} + 1 \right)^{2n + \frac{3}{2}} = 
\frac{e \, A^3}{2^{\frac{25}{12}} \, \pi^{\frac{1}{4}}},
\end{equation*} 
\begin{equation*}\label{variant3n+1}
\prod_{n \geq 1} \frac{1}{e}
\left( \frac{1}{3n+1} + 1 \right)^{3n + \frac{3}{2}} = 
\frac{e^{1 + \frac{3 \psi\,'\left(\frac{1}{3} \right) - 2 \pi^2 }{6 \sqrt{3} \pi}}}{2^{3/2} 3^{1/4}}, 
\end{equation*}
\begin{equation*}\label{variant4n+1} 
\prod_{n \geq 1} \frac{1}{e} \left(\frac{1}{4 n+1}+1\right)^{4 n+\frac{3}{2}}
= \frac{e^{1+\frac{\cat}{\pi }} \sqrt{\Gamma \left(\frac{1}{4}\right)}}{2^{11/6} A^{3/2} \sqrt[4]{\pi}}. 
\end{equation*}
\end{example}


\begin{remark}
 Another proof of the equalities in Example~\ref{multi} can be deduced from two formulas involving the Riemann $\zeta$ function that 
 can be found, e.g., in the book of Srivastava and Choi (\cite[Equations~(10) p.\ 270 and~(482) p.\ 311]{SrivastavaChoi2012}),
namely, for $|t| < 1$,
\begin{align*}
\sum_{k \geq 2} (-1)^k \zeta(k) \frac{t^k}{k} \,
&= \, \log \Gamma(1+t) + \EMcst t, \\
\sum_{k \geq 2} (-1)^k \zeta(k) \frac{t^{k+1}}{k+1} \,
&= \,
(1 - \log(2\pi))\, \frac{t}{2} + \frac{1+\EMcst}{2} \, t^2 
+ \log \barnes(1+t).
\end{align*} 
\end{remark}

\section{The Borwein--Dykshoorn function}
\label{sec: Borwein--Dykshoorn function}

 In the paper \cite{BorweinDykshoorn1993}, Borwein and Dykshoorn introduced the function $$ D(x) = \lim_{n \to + \infty} \prod_{k = 
 1}^{2n+1} \left(1+\frac{x}{k}\right)^{k(-1)^{k+1}} $$ as a generalization of a result of Melzak \cite{Melzak1961}
who proved that 
$$
\lim_{n \to + \infty} \prod_{k=1}^{2n+1} \left(1+\frac{2}{k}\right)^{k(-1)^{k+1}} = \frac{\pi e}{2}.
$$
Using Equation~\eqref{variant2} and the expression (see, e.g., 
\cite{Allouche2015}),
$$
D(x) = e^x\prod_{k \geq 1}\left(e^{-x} 
\left(1 + \dfrac{x}{k}\right)^k\right)^{(-1)^{k+1}},
$$ 
we could apply Theorem~\ref{B-alpha-beta} to recover the 
expression of $D(x)$ in terms of the Barnes and the Gamma functions 
given\footnote{Note that the left-hand side of Equation (46) in 
\cite{{Adamchik2005}} is not quite $D(x)$.} in 
\cite[Proposition~5, p.~284]{Adamchik2005}, but we 
limit ourselves to apply Theorem~\ref{B-alpha-beta} to obtain 
the equality $D(1) = \frac{A^6}{2^{1/6} \sqrt{\pi}}$ 
proven in \cite[p.~204]{BorweinDykshoorn1993}. Actually, 
we prove the (equivalent) formulation due to Sondow and 
Hadjicostas \cite[p.~312]{SondowHadjicostas2007}:
\begin{equation}
\label{eq:SondowHadjicostas2007}
 \prod_{n \geq 1} 
\left(\frac{e}{\left( 1 + \frac{1}{n} \right)^{n}}\right)^{(-1)^{n-1}}
= \ \ \frac{2^{\frac{1}{6}} e \sqrt{\pi}}{A^6}, 
\end{equation}
where the left-hand side is $e/D(1)$. 
Namely, write
$$
\begin{array}{lll}
\displaystyle \prod_{n \geq 1} 
\left(\frac{e}{\left(1 + \frac{1}{n} \right)^{n}}\right)^{(-1)^{n-1}}
&=& \displaystyle\prod_{n \geq 1} 
\frac{\left(1 + \frac{1}{2n}\right)^{2n}}
{\left(1 + \frac{1}{2n-1}\right)^{2n-1}} \\
&=& \displaystyle\prod_{n \geq 1} 
\left(\frac{\frac{1}{e^2}\left(1+\frac{1}{2n}\right)^{4n+1}}
{\frac{1}{e}\left(1 + \frac{1}{2n-1}\right)^{2n-\frac{1}{2}}
\frac{1}{e} \left(1+\frac{1}{2n}\right)^{2n+\frac{1}{2}}} 
\cdot R_n\right),
\end{array}
$$
where
$$
R_n := \frac{\left(1+\frac{1}{2n-1}\right)^{\frac{1}{2}}
\left(1+\frac{1}{2n}\right)^{\frac{1}{2}}}
{\left(1+\frac{1}{2n}\right)}.
$$
Carefully checking that all involved infinite products do converge
yields
$$
\begin{array}{lll}
\displaystyle \prod_{n \geq 1} 
\left(\frac{e}{\left(1 + \frac{1}{n} \right)^{n}}\right)^{(-1)^{n-1}}
&=&
\displaystyle \frac{ 
\left(\prod\limits_{n \geq 1}\frac{1}{e}\left(1+\frac{1}{2n}\right)^{2n+\frac{1}{2}}\right)^2
\prod\limits_{n \geq 1} R_n}
{\prod\limits_{n \geq 1} \frac{1}{e}\left(1 + \frac{1}{2n-1}\right)^{2n-\frac{1}{2}} 
\prod\limits_{n \geq 1} \frac{1}{e} \left(1+\frac{1}{2n}\right)^{2n+\frac{1}{2}}
}.
\end{array}
$$
Thus, using the second identity in Example~\ref{ex:ex of products} in the first term of the numerator and assembling the two factors in the denominator, we have
$$
\prod_{n \geq 1} 
\left(\frac{e}{\left(1 + \frac{1}{n} \right)^{n}}\right)^{(-1)^{n-1}}
= \
\frac{\dfrac{2^{\frac{1}{6}} e^2}{\pi^{\frac{1}{2}} A^6}\prod\limits_{n \geq 1} R_n}
{\prod\limits_{n \geq 1}\frac{1}{e}\left(1+\frac{1}{n}\right)^{n+\frac{1}{2}}}. 
$$
Remark~\ref{eq: special value for 1, a half and 1} thus gives
\[
\prod_{n \geq 1} \left(\frac{e}{\left(1 + \frac{1}{n} \right)^{n}}\right)^{(-1)^{n-1}}
= \ \frac{\dfrac{2^{\frac{1}{6}} e^2}{\pi^{\frac{1}{2}} A^6}\prod\limits_{n \geq 1} R_n}
{\frac{e}{\sqrt{2\pi}}}.
\]
Now 
$$
\prod_{n \geq 1} R_n = \sqrt{\prod_{n \geq 1} 
\frac{(2n)(2n)}{(2n-1)(2n+1)}} =
\sqrt{\Gamma\left(\frac{1}{2}\right) \Gamma\left(\frac{3}{2}\right)}
= \sqrt{\frac{\pi}{2}},
$$
where we used Lemma~\ref{lem:gammapro}, Relation~\eqref{eq:Gamma+1}, 
and the value $\Gamma(1/2) = \sqrt{\pi}$. Putting the last two 
equalities together gives Equation~\eqref{eq:SondowHadjicostas2007}.

\section{Relation to Sondow--Hadjicostas' work}\label{sectionSH}
 In this section, first we show how the work of Sondow and Hadjicostas   \cite{SondowHadjicostas2007} can be used to evaluate products of  
  a  form similar to, but different from, Gosper's Equation~\eqref{Gosper},
thus unifying some proofs that can be found in the literature. Then we 
show two approaches for proving Gosper's Equation~\eqref{Gosper} by 
using results in \cite{SondowHadjicostas2007}.

\subsection{From Sondow--Hadjicostas to Gosper-like products}
\label{GSH}

First we recall a theorem due to Sondow and Hadjicostas.
Then we give several applications.

\begin{theorem}[{\cite[Thm.~15]{SondowHadjicostas2007}}]
\label{thm:Sondow-Hadjicostas}
Let $(\theta_n)_{n\ge 1}$ be a periodic sequence in $\C$ 
and let $q$ be a positive multiple of its period.
If $\sum_{j=1}^q \theta_j = 0$, then
 \begin{equation*}
 \label{eq:Sondow-Hadjicostas}
 \sum_{n\ge 1} \theta_n \left( 1 - n\log \frac{n+1}{n} \right) 
 =
 \sum_{j=1}^{q} \theta_j \log \frac{K(\frac{j+1}{q})^q e^{-j/q}}{K(\frac{j}{q})^q \Gamma(\frac{j+1}{q})}.
 \end{equation*}
\end{theorem}

\begin{corollary}\label{cor:applications}
As in Theorem~\ref{thm:Sondow-Hadjicostas}, let
$(\theta_n)_{n\ge 1}$ be a sequence of 
reals of period $q$ such that $\sum_{j=1}^q \theta_j = 0$. 
Then, defining $\theta_0 := \theta_q$, we have
\begin{align*}
&\prod_{n\ge 1} 
\prod_{j=0}^{q-1}\left(\frac{qn+j+1}{qn+j}\right)^{\theta_j(qn+j)} \\
&= e^{-\theta_0} \prod_{n=1}^{q-1} 
\left(\frac{n}{n+1}\right)^{n \theta_n} 
\prod_{j=1}^q \left(\frac{K(\frac{j}{q})^q \Gamma(\frac{j+1}{q})}{K(\frac{j+1}{q})^q e^{-j/q}} \right)^{\theta_j} \\
&= e^{-\theta_0} \prod_{n=1}^{q-1} \left(\frac{n}{n+1}\right)^{n \theta_n} 
\prod_{j=1}^q \left(\frac{G_B(\frac{j+1}{q})^q \Gamma(\frac{j+1}{q})^{q-j}}{G_B(\frac{j}{q})^q \Gamma(\frac{j}{q})^{q-j} e^{-j/q}}\right)^{\theta_j}.
\end{align*}
\end{corollary}
\begin{proof}
We prove the first equality of the statement.
Since the general term of the series in Theorem~\ref{thm:Sondow-Hadjicostas} tends to zero, we can group the terms in each period, and (after extending the
sequence $(\theta_n)_{n\ge 1}$ by setting $\theta_0 := \theta_q$ as in the statement), we may write
\begin{align*}
 & \sum_{n\ge 1} \theta_n \left( 1 - n\log \frac{n+1}{n} \right) \\ 
&= 
 \displaystyle\sum_{n=1}^{q-1} \theta_n \left( 1 - n\log \frac{n+1}{n} \right) + \displaystyle\sum_{n\ge 1} \sum_{j=0}^{q-1}
\theta_{qn+j} \left( 1 - (qn+j)\log \frac{qn+j+1}{qn+j} \right)
\\
&= 
 \displaystyle\sum_{n=1}^{q-1} \theta_n \left( 1 - n\log \frac{n+1}{n} \right) 
 + \displaystyle\sum_{n\ge 1} \sum_{j=0}^{q-1} \theta_j
 \left(1 - (qn+j)\log \frac{qn+j+1}{qn+j} \right),
\end{align*}
where we have used the fact that $q$ is a multiple of the period of $(\theta_n)_{n\ge 1}$. Now using the assumption that $\sum_{j=0}^{q-1} \theta_j = 0$ twice, we obtain
\begin{align*}
& \sum_{n\ge 1} \theta_n \left( 1 - n\log \frac{n+1}{n} \right) \\
& =
- \theta_0 
 - \displaystyle\sum_{n=1}^{q-1} \left(n \theta_n \log \frac{n+1}{n} \right) 
 - \displaystyle\sum_{n\ge 1} \sum_{j=0}^{q-1}
 \left(\theta_j (qn+j)\log \frac{qn+j+1}{qn+j} \right).
\end{align*}
Next, using Theorem~\ref{thm:Sondow-Hadjicostas} (while paying attention to the negative sign above) and exponentiating give
\[
\prod_{j=1}^q \left(\frac{K(\frac{j}{q})^q \Gamma(\frac{j+1}{q})}{K(\frac{j+1}{q})^q e^{-j/q}} \right)^{\theta_j}
=
e^{\theta_0} \prod_{n=1}^{q-1} \left(\frac{n+1}{n}\right)^{n \theta_n} 
\prod_{n\ge 1} \prod_{j=0}^{q-1}\left(\frac{qn+j+1}{qn+j}\right)^{\theta_j(qn+j)}.
\]
The second equality in the statement can be obtained using the relationship in Equation~\eqref{eq:relation-K-Gamma-GB}.
\end{proof}

\medskip

An example of application of this result in 
\cite{SondowHadjicostas2007} gives a proof of 
Equation~\eqref{eq:SondowHadjicostas2007} from 
the previous section.

\begin{example}[{\cite[Ex.~1]{SondowHadjicostas2007}}]
We have
 \[
 \prod_{n\ge 1} \left(\frac{e}{(1+\frac{1}{n})^n}\right)^{(-1)^{n-1}} = \frac{2^{\frac{1}{6}}e\sqrt{\pi}}{A^6}.
 \]
To show this, as explained in~\cite[Ex.~11]{SondowHadjicostas2007}, one can consider the sequence $\theta_n=(-1)^n$ and $q=2$ in Theorem~\ref{thm:Sondow-Hadjicostas}. 
\end{example}

Other special cases give results that can be found in the literature. For example, we have the following.

\begin{proposition}[{\cite[Eq.~(48) and Eq.~(49)]{Campbell2024Automatic}}] \label{pro:campbell48}
 We have  that  
 \[ 
 Q_1 := \prod_{n\ge 1} 
 \frac{(n+\frac{1}{2})^{8n+4}}{(n+1)^{4n+3} n^{4n+1}} = \frac{16\sqrt[3]{2}}{A^{12}} 
\ \ \text{and} \ \ 
Q_2 := \prod_{n\ge 1} 
\frac{(1+\frac{1}{4n})^{8n+1}}{(1+\frac{1}{4n+1})^{8n+3}} 
= \frac{2^{\frac{14}{3}}}{A^6 \sqrt{\pi} e^{\frac{4\cat}{\pi}}}.
\]
\end{proposition}

\begin{proof}
We first treat $Q_1$.
We have  that 
\[
Q_1 \cdot \prod_{n\ge 1}
\frac{n(n+\frac{3}{4})}{(n+\frac{1}{4})(n+\frac{1}{2})} 
= \prod_{n\ge 1} 
\frac{(4n+2)^{8n+4}}{(4n+4)^{4n+3}(4n)^{4n+1}}
\cdot 
\prod_{n\ge 1}
\frac{4n(4n+3)}{(4n+1)(4n+2)}
\]
\begin{equation}\label{Ptimes}
= \prod_{n\ge 1} \left( \frac{4n+1}{4n} \right)^{4n} \left( \frac{4n+2}{4n+1} \right)^{4n+1} \left( \frac{4n+2}{4n+3} \right)^{4n+2} \left( \frac{4n+3}{4n+4} \right)^{4n+3}.
\end{equation}
In Corollary~\ref{cor:applications}, take $q=4$ and the sequence 
$(\theta_n)_{n\ge 1}$ defined by $\theta_{4n}=1$ for all $n\ge 1$, 
and $\theta_{4n+1}=1$, $\theta_{4n+2}=-1$, and $\theta_{4n+3}=-1$ 
for all $n\ge 0$, and also set $\theta_0=\theta_4=1$. 
Then Corollary~\ref{cor:applications} gives 
\begin{align*}
& \hskip -.5truecm \prod_{n\ge 1} \left(\frac{4n+1}{4n}\right)^{4n} \left(\frac{4n+2}{4n+1}\right)^{4n+1} 
\left(\frac{4n+2}{4n+3}\right)^{4n+2} 
\left(\frac{4n+3}{4n+4} \right)^{4n+3} \\
&=
\left(\frac{8}{3e}\right)
\left(\frac{K(\frac{1}{4})^4 \Gamma(\frac{1}{2})}{K(\frac{1}{2})^4 
e^{-1/4}} \cdot 
\frac{K(\frac{3}{4})^4 e^{-1/2}}{K(\frac{1}{2})^4 \Gamma(\frac{3}{4})} \cdot 
\frac{K(1)^4 e^{-3/4}}{K(\frac{3}{4})^4 \Gamma(1)} \cdot \frac{K(1)^4 \Gamma(\frac{5}{4})} {K(\frac{5}{4})^4 e^{-1}}\right) 
\\ 
&= 
\left(\frac{8}{3e}\right)
\left(\frac{K(\frac{1}{4})^4 K(1)^8 \Gamma(\frac{1}{2}) \Gamma(\frac{5}{4})}{K(\frac{1}{2})^8 K(\frac{5}{4})^4 \Gamma(\frac{3}{4}) \Gamma(1)}\right) = 
\frac{8\Gamma(\frac{1}{4})^2 2^{1/3}}{3\sqrt{{2\pi}}A^{12}},
\end{align*}
where, for the last equality, we have used the following values
\begin{align}
\label{eq:special values for gamma and K}
&\Gamma(1)= K(1) = 1, 
\ \ \Gamma\left(\frac{1}{2}\right) = \sqrt{\pi}, 
\ \ \Gamma\left(\frac{5}{4}\right) = \frac{1}{4}\Gamma\left(\frac{1}{4}\right), 
\ \ \ \ K\left(\frac{1}{2}\right) = \frac{A^{\frac{3}{2}}}{2^{\frac{1}{24}} e^{\frac{1}{8}}}, \nonumber \\ 
&K\left(\frac{5} {4}\right) = \frac{1}{4^{1/4}} K\left(\frac{1}{4}\right),
\ \ \Gamma\left(\frac{1}{4}\right)\Gamma\left(\frac{3}{4}\right) = \frac{\pi}{\sin(\pi/4)} = \pi \sqrt{2}.
\end{align}
(see Equations~\eqref{eq:Gamma+1},~\eqref{eq:formula for prod of gammas} and~\eqref{eq:K+1}).
Hence Equation~\eqref{Ptimes} gives
\[
Q_1 \cdot \prod_{n\ge 1}
\frac{n(n+\frac{3}{4})}{(n+\frac{1}{4})(n+\frac{1}{2})} 
= \frac{8\Gamma(\frac{1}{4})^2 2^{1/3}}{3\sqrt{{2\pi}}A^{12}}.
\]
However, Lemma~\ref{lem:gammapro} yields
\[
\prod_{n\ge 1}
\frac{n(n+\frac{3}{4})}{(n+\frac{1}{4})(n+\frac{1}{2})} 
=\frac{\Gamma\left(1+\frac{1}{4}\right) \Gamma\left(1+\frac{1}{2}\right)}{\Gamma(1) \Gamma\left(1+\frac{3}{4}\right)}
= \frac{\Gamma\left(\frac{1}{4}\right)\Gamma\left(\frac{1}{2}\right)}{6\ \Gamma\left(\frac{3}{4}\right)} 
= \frac{\Gamma(\frac{1}{4})^2}{6\sqrt{2\pi}},
\]
where we again used Equation~\eqref{eq:Gamma+1} and the special values highlighted in Equation~\eqref{eq:special values for gamma and K}.
Hence
\[
Q_1 = \frac{8\Gamma(\frac{1}{4})^2 2^{1/3}}{3\sqrt{{2\pi}}A^{12}} \cdot 
\frac{6\sqrt{2\pi}}{\Gamma(\frac{1}{4})^2}
= \frac{16\sqrt[3]{2}}{A^{12}}, 
\]
as desired.

\bigskip

Similarly, we write 
\begin{align*}
Q_2 \cdot \prod_{n\ge 1}
\frac{n(n+\frac{1}{2})}{(n+\frac{1}{4})^2}
&=\prod_{n\ge 1} 
\frac{(4n+1)^{16n+4}}{(4n)^{8n+1}(4n+2)^{8n+3}}
\cdot \prod_{n\ge 1}
\frac{(4n)(4n+2)}{(4n+1)^2} \\
& \!\!\!\!\! = \prod_{n\ge 1} 
\frac{(4n+1)^{16n+2}}{(4n)^{8n}(4n+2)^{8n+2}} 
= \left(\prod_{n\ge 1}
\frac{(4n+1)^{8n+1}}{(4n)^{4n}(4n+2)^{4n+1}}\right)^2 .
\end{align*}
In Corollary~\ref{cor:applications}, take $q=4$ and the sequence 
$(\theta_n)_{n\ge 1}$ defined by $\theta_{4n}=1$ for all $n\ge 1$, 
and $\theta_{4n+1}=-1$, $\theta_{4n+2}=0$, and $\theta_{4n+3}=0$ 
for all $n\ge 0$, and set $\theta_0=\theta_4=1$.
Then Corollary~\ref{cor:applications} gives 
\begin{align*}
\prod_{n\ge 1}
\frac{(4n+1)^{8n+1}}{(4n)^{4n}(4n+2)^{4n+1}} &=
\prod_{n\ge 1} \left(\frac{4n+1}{4n} \right)^{4n} 
\left( \frac{4n+1}{4n+2} \right)^{4n+1} \\
&= \frac{2}{e} \left(\frac{K(\frac{1}{4})^4 \Gamma(\frac{1}{2})}{K(\frac{1}{2})^4 e^{-\frac{1}{4}}}\right)^{-1}
\left(\frac{K(1)^4 \Gamma(\frac{5}{4})} {K(\frac{5}{4})^4 e^{-1}}\right) \\
&= \frac{2^{5/6} A^6 \Gamma(\frac{1}{4})}{e^{\frac{3}{4}} \sqrt{\pi}K(\frac{1}{4})^8},
\end{align*}
where we used the special values in Equation~\eqref{eq:special values for gamma and K} above.
Thus
\[
Q_2 \cdot \prod_{n\ge 1}
\frac{n(n+\frac{1}{2})}{(n+\frac{1}{4})^2} =
\left(\frac{2^{5/6} A^6 \Gamma(\frac{1}{4})}{e^{\frac{3}{4}}
\sqrt{\pi}K(\frac{1}{4})^8}\right)^2
= \frac{2^{5/3} A^{12} \Gamma(\frac{1}{4})^2}{e^{\frac{3}{2}}\pi 
K(\frac{1}{4})^{16}}.
\]
Using Lemma~\ref{lem:gammapro}, we have
\[
\prod_{n\ge 1}
\frac{n(n+\frac{1}{2})}{(n+\frac{1}{4})^2} 
= \frac{\Gamma(\frac{5}{4})^2}{\Gamma(1)\Gamma(\frac{3}{2})}
= \frac{\frac{1}{16}\Gamma(\frac{1}{4})^2}{\frac{1}{2}\Gamma(\frac{1}{2})}
= \frac{\Gamma(\frac{1}{4})^2}{2^3\Gamma(\frac{1}{2})}
= \frac{\Gamma(\frac{1}{4})^2}{8 \sqrt{\pi}},
\]
where we used the special value $\Gamma(1/2)=\sqrt{\pi}$.
Hence, using the value of $K(1/4)$ given 
in Table \ref{tab:K-special-values}, we obtain
\[
Q_2 = \left(\frac{\Gamma(\frac{1}{4})^2}{8 \sqrt{\pi}}\right)^{-1} 
\cdot \frac{2^{5/3} A^{12} \Gamma(\frac{1}{4})^2}{e^{\frac{3}{2}}\pi 
K(\frac{1}{4})^{16}}
=\frac{2^{\frac{14}{3}}A^{12}}{e^{\frac{3}{2}}\sqrt{\pi} A^{18} e^{4\frac{\cat}{\pi}-\frac{3}{2}}.} 
= \frac{2^{\frac{14}{3}}}{A^6 \sqrt{\pi} e^{\frac{4\cat}{\pi}}},
\]
as desired.
\end{proof}

\bigskip

Now let us consider two examples inspired by Gosper's 
Equation~\eqref{Gosper} with four and with three factors instead of two.
Note that we will use Corollary~\ref{cor:applications}---hence the 
results of \cite{SondowHadjicostas2007}---for these two examples, but
they can both be immediately and directly evaluated by {\tt Mathematica}.

\begin{proposition}\label{almost-Gosper}
We have
\[
R := \prod_{n\ge 1} \frac{n^{6n} \, (n+1)^{3n+2}}{(n+\frac{2}{3})\, (n+\frac{1}{3})^{9n+1}}
= \frac{4 \pi \sqrt{3} \, A^{12}}{3^{\frac{25}{8}} \, \Gamma(\frac{1}{3})^3} \, e^{\frac{3\psi\,'(\frac{1}{3})-2\pi^2}{4\pi \sqrt{3}}}
\]
and
\[
S := \prod_{n\ge 1}
\frac{(n+\frac{1}{3})^{6n+1}}{n^{3n} \, (n+\frac{2}{3})^{3n+1}}
= \frac{3^{\frac{13}{24}} \, \Gamma^2(\frac{1}{3}) 
\, e^{\frac{\pi \sqrt{3}}{6} - \psi\,'(\frac{1}{3}) \frac{\sqrt{3}}{4 \pi}}}{\pi A^4}.
\]
\end{proposition}
\begin{proof}
First, we write
\[
R = \prod_{n\ge 1}
\frac{(3n)^{6n} \, (3n+3)^{3n+2}}{(3n+2) \, (3n+1)^{9n+1}}
= \prod_{n\ge 1} \left( \frac{3n}{3n+1} \right)^{6n} \left( \frac{3n+2}{3n+1} \right)^{3n+1} \left( \frac{3n+3}{3n+2} \right)^{3n+2}.
\]
Then, in Corollary~\ref{cor:applications}, take $q=3$ and define
the sequence $(\theta_n)_{n\ge 1}$ by $\theta_{3n} = -2$ for all 
$n\ge 1$, and $\theta_{3n+1} = \theta_{3n+2} = 1$ for all 
$n\ge 0$, and set $\theta_0=1$. This yields
\begin{align*}
R & = \prod_{n\ge 1} \left( \frac{3n+1}{3n} \right)^{-6n} 
\left( \frac{3n+2}{3n+1} \right)^{3n+1} 
\left( \frac{3n+3}{3n+2} \right)^{3n+2} 
\\
&=
\left(\frac{2e^2}{9}\right)
\left(\frac{K(\frac{1}{3})^3 \Gamma(\frac{2}{3})}{K(\frac{2}{3})^3 e^{-\frac{1}{3}}}\right)
\left(\frac{K(\frac{2}{3})^3 \Gamma(1)}{K(1)^3 e^{-\frac{2}{3}}}\right)
\left(\frac{K(1)^3 \Gamma(\frac{4}{3})}{K(\frac{4}{3})^3 e^{-1}} \right)^{-2} \\
&= \frac{2A^{12}}{3^{\frac{17}{8}}}
\frac{\Gamma(\frac{2}{3})}{\Gamma(\frac{1}{3})^2} e^{\frac{3\psi\,'(\frac{1}{3})-2\pi^2}{4\pi \sqrt{3}}}
= \frac{4 \pi \sqrt{3} A^{12}}{3^{\frac{25}{8}}\Gamma(\frac{1}{3})^3}
e^{\frac{3\psi\,'(\frac{1}{3})-2\pi^2}{4\pi \sqrt{3}}}, 
\end{align*}
 where we have the 
value of $K\left( \frac{1}{3} \right)$ in Table \ref{tab:K-special-values}, Formula~\eqref{eq:K+1},
and the well-known Relations~\eqref{eq:Gamma+1} 
and~\eqref{eq:formula for prod of gammas}. 

\bigskip

Next, we write
\[
S = \prod_{n\ge 1} 
\frac{(n+\frac{1}{3})^{6n+1}}{n^{3n}(n+\frac{2}{3})^{3n+1}} =
\prod_{n\ge 1} \frac{(3n+1)^{6n+1}}{(3n)^{3n}(3n+2)^{3n+1}}.
\]
Then, in Corollary~\ref{cor:applications}, take $q=3$ and the sequence 
$(\theta_n)_{n\ge 1}$ defined by $\theta_{3n} = 1$ for all 
$n\ge 1$, $\theta_{3n+1} = -1$, and $\theta_{3n+2} = 0$ 
for all $n\ge 0$, and set $\theta_0=1$. We obtain
\begin{align*}
S & = \prod_{n\ge 1}\left(\frac{3n+1}{3n}\right)^{3n}
\left(\frac{3n+2}{3n+1}\right)^{-(3n+1)} 
= \left(\frac{2}{e}\right)
\left(\frac{K(\frac{1}{3})^3 \Gamma(\frac{2}{3})}{K(\frac{2}{3})^3 e^{-\frac{1}{3}}}\right)^{-1}
\left(\frac{K(1)^3\Gamma(\frac{4}{3})}{K\left(\frac{4}{3}\right)^3 e^{-1}}\right) \\
&= 2 \ e^{-\frac{1}{3}}\left(\frac{K\left(\frac{2}{3}\right)^3 \Gamma(\frac{4}{3})}{K\left(\frac{1}{3}\right)^3 K\left(\frac{4}{3}\right)^3 \Gamma\left(\frac{2}{3}\right)}\right) 
= 2 \ e^{-\frac{1}{3}} \left(\frac{K\left(\frac{2}{3}\right)^3 \Gamma\left(\frac{1}{3}\right)} {K\left(\frac{1}{3}\right)^6\Gamma\left(\frac{2}{3}\right)} \right),
\end{align*}
where we have used Equations~\eqref{eq:Gamma+1} and~\eqref{eq:K+1}.
But, taking $n=3$ in Lemma~\ref{lem:relation-K-A}, we have
\[
K\left(\frac{1}{3}\right)K\left(\frac{2}{3}\right) =
\frac{A^{\frac{8}{3}}}{3^{\frac{1}{36}}e^{\frac{2}{9}}}
\ \ \ \text{hence} \ \ \
K\left(\frac{2}{3}\right)^3 = \frac{A^8}{3^{\frac{1}{12}}
e^{\frac{2}{3}}K\left(\frac{1}{3}\right)^3}.
\]
Furthermore, with Equation~\eqref{eq:formula for prod of gammas}, we have
\[
\Gamma\left(\frac{1}{3}\right) \Gamma\left(\frac{2}{3}\right)
= \frac{\pi}{\sin(\frac{\pi}{3})} = \frac{2\pi}{\sqrt{3}}.
\]
This yields
\[
S = \frac{\sqrt{3}}{e\pi} 
\left(\frac{A^8 \Gamma^2(\frac{1}{3})}{3^{\frac{1}{12}} K(\frac{1}{3})^9}\right)
= \frac{3^{\frac{13}{24}} \Gamma^2(\frac{1}{3}) e^{\frac{\pi \sqrt{3}}{6} - \psi\,'(\frac{1}{3}) \frac{\sqrt{3}}{4 \pi}}}{\pi A^4},
\]
 where we have used the value of $K\left(\frac{1}{3}\right)$
given in Table \ref{tab:K-special-values}.
\end{proof}

\subsection{Gosper's product via Sondow-Hadjicostas' results}\label{Gosper-Sondow-Hadjicostas}
  t is possible to prove the value of Gosper's product by using
the results of Sondow-Hadjicostas. We give two approaches, one takes
advantage of the results of the previous section, the other one is
more directly deduced from \cite{SondowHadjicostas2007}.

\subsubsection*{First approach}

As in Proposition~\ref{almost-Gosper}, we have seen that
\begin{align*}
R &:= \prod_{n\ge 1} \frac{n^{6n}(n+1)^{3n+2}}{(n+\frac{2}{3})(n+\frac{1}{3})^{9n+1}}
= \prod_{n\ge 1}
\frac{(3n)^{6n}(3n+3)^{3n+2}}{(3n+2)(3n+1)^{9n+1}} \\
&= \frac{4 \pi \sqrt{3} A^{12}}{3^{\frac{25}{8}}\Gamma(\frac{1}{3})^3}
e^{\frac{3\psi\,'(\frac{1}{3})-2\pi^2}{4\pi \sqrt{3}}}, 
\end{align*}
but, in Remark~\ref{eq: special value for 1, a half and 1}, we have
seen the equality
$$
T := \prod_{n \geq 1} \frac{1}{e}
\left(\frac{1}{n}+1\right)^{n + \frac{1}{2}}
= 
\prod_{n \geq 1} \frac{1}{e} \left(\frac{3n+3}{3n}\right)^{n + \frac{1}{2}}
= \frac{e}{\sqrt{2 \pi}}
$$
(of which we have given a quite elementary proof).

\medskip

Now compute $\dfrac{T^3}{R}$:
\begin{align*}
\frac{T^3}{R} &= \prod_{n \geq 1} \frac{1}{e^3}
\left(\frac{3n+3}{3n}\right)^{3n + \frac{3}{2}} \,
\left(\frac{(3n+2)(3n+1)^{9n+1}}{(3n)^{6n}(3n+3)^{3n+2}}\right) \\
&= \prod_{n \geq 1} \frac{1}{e^3} 
\left(\frac{3n+1}{3n}\right)^{9n + \frac{3}{2}}
\left(\frac{(3n+2)^2}{(3n+3)(3n+1)}\right)^{\frac{1}{2}} \\
&= \prod_{n \geq 1} \frac{1}{e^3} 
\left(\frac{3n+1}{3n}\right)^{9n + \frac{3}{2}} \,
\left(\prod_{n \geq 1} \frac{(3n+2)^2}{(3n+3)(3n+1)}\right)^{\frac{1}{2}}. 
\end{align*}
Now we see that the first product is the cube of the Gosper constant,
while the second, by using Lemma~\ref{lem:gammapro} and the reflection
formula~(\ref{eq:formula for prod of gammas}), is equal to
the square root of
$$
\frac{\Gamma(2) \Gamma\left(\frac{4}{3}\right)}
{\Gamma\left(\frac{5}{3}\right)^2} = 
\frac{3\Gamma\left(\frac{1}{3}\right)}{4\Gamma\left(\frac{2}{3}\right)^2} 
= \frac{9}{16\pi^2} \, \Gamma\left(\frac{1}{3}\right)^3.
$$
This implies
$$
\prod_{n \geq 1} \frac{1}{e^3} 
\left(\frac{3n+1}{3n}\right)^{9n + \frac{3}{2}} =
e^3\frac{\sqrt{2\pi}\, T^3}{3\, \Gamma\left(\frac{1}{3}\right)^{\frac{3}{2}}\, R} 
$$
so that, after substituting the value of $R$, we obtain Gosper's Equation~(\ref{Gosper}).

\subsubsection*{Second approach}

We still let $P$ denote Gosper's product~\eqref{Gosper}. 
Define $b_n = -1 + \displaystyle\left(n + \frac{1}{2}\right) 
\log \frac{n+1}{n}$ for all $n\ge 1$.
Then $\log P = \displaystyle\sum_{n \geq 1} b_{3n}$.

 \ 


 \ 

Let $\omega := e^{\frac{2i\pi}{3}}$ and let, 
following\footnote{Note that the function $f_1$ 
introduced at about the same time in 
\cite[Eq.~(6)]{Pilehrood-Pilehrood2007} 
satisfies $f_1(z) =z\gamma(z)$. Also note that the notation
$\gamma(z)$ is justified by the value $\gamma(1) = \gamma$.}
\cite{SondowHadjicostas2007},
$$
\gamma(z) := 
\sum_{n \geq 1} z^{n-1} 
\left(\frac{1}{n} - \log \frac{n+1}{n} \right), 
$$
 for complex $z$ such that $|z| \leq 1$. 

We have
$$
\log P = \gamma(1) + \omega \gamma(\omega) + \omega^2 \gamma(\omega^2).
$$
Then it suffices to use the expressions of $\gamma(z)$ given in 
\cite{SondowHadjicostas2007} when $z$ is a root of unity.

\section{Higher-degree extensions}\label{sec:higherorder}
Gosper's product~\eqref{Gosper} leads us to consider 
extensions of the form 
\begin{equation}\label{commongen}
 \prod_{n\ge 1} e^{r(n, x)} \left( 1 + \frac{x}{p_1(n)} \right)^{p_2(n)} 
\end{equation}
for a complex variable $x$ and for a rational function $r(n, x)$ 
in $n$ and $x$ and for univariate polynomials $p_1(n)$ and $p_2(n)$
of equal degree. Since $ p_1(n)$ and $p_2(n)$ are both affine
polynomials for the case of Gosper's product, this raises questions 
as to extensions and analogues involving 
quadratic polynomials. 

Let $\phi = \frac{1 + \sqrt{5}}{2}$ denote the golden 
ratio. The work of Bradley--Thrush \cite{BradleyThrush2025} 
concerning Ramanujan's product was largely motivated by the 
corrected version 
 \begin{equation}\label{Ramanujangolden}
 \prod_{n\ge 1} e^{-\frac{\log^2 \phi}{\pi^2}} 
 \left( 1 + \frac{\log^2 \phi}{n^2 \pi^2} \right)^{n^{2}} 
 = e^{\frac{ \log^2 \phi }{2\pi^2} - \frac{\zeta(3)}{10\pi^2}} 
 \end{equation}
 given by Bradley--Thrush of a purported 
 product evaluation from the second volume of Ramanujan's notebooks
 \cite[p.\ 365]{Ramanujan1957}. 
 The functional equations-based approach from Bradley--Thrush
 yielded new and Ramanujan-like infinite product evaluations such as 
 \begin{equation}\label{BTgolden}
 \prod_{n\ge 1} \sqrt{\frac{2\pi}{n}} \frac{ \left( n \phi^2 \right)^{n \phi^2} 
 e^{-n \phi^2 + \frac{1}{12 n \phi^2}} }{\phi \Gamma(n \phi^2)} = 
 \frac{e^{\frac{1}{12} \EMcst \left(3-\phi ^2\right)} \phi ^{-\frac{\phi ^2}{6}} A^{\phi ^2}}{\sqrt[4]{2 \pi }}, 
\end{equation}
but the general problem of evaluating products as in Equation~\eqref{commongen}, for a free variable $x$, was not considered. 

A recent paper by Allouche and Markov \cite{AlloucheMarkov2026}    introduced techniques and results related to the products in Equations~\eqref{Ramanujangolden} and \eqref{BTgolden} 
and to a remarkable identity due to Ramanujan, see~\cite[p.~364]{Ramanujan1957} or \cite[p.~488]{Berndt1998}, stating that, for $x > 0$,
\begin{multline}\label{Ramanujan}
 \prod_{n \geq 1} e^{-x^{2}} \left( 1 + \frac{x^2}{n^2} \right)^{n^{2}} = 
 \left( 1 - e^{-2\pi x} \right)^{-x^2} \cdot \\ 
 \exp\left( 
 \frac{x^2}{2} - \frac{\pi x^3}{3} + \frac{x}{\pi} \operatorname{Li}_{2}(e^{-2\pi x}) 
 + \frac{1}{2 \pi^2} \operatorname{Li}_{3}(e^{-2\pi x}) 
 - \frac{\zeta(3)}{2\pi^2} \right). 
\end{multline}

Another motivation for the study of higher degree generalizations
of Gosper's identity is the theory of multisines and multicosines
functions due to Kurokawa and Rovinski\u{\i} (see \cite{Kurokawa1991, 
Kurokawa1992, Rovinskii1991}). 
In particular, the triple sine and the triple cosine are given by 
(see, e.g., \cite{KurokawaWakayama2003}):
$$ {\mathcal S}_3(x) =
e^{\frac{x^2}{2}} \prod_{n \geq 1} \left(e^{x^2}\left(1 - \frac{x^2}{n^2}\right)^{n^2}\right),
$$
$$
{\mathcal C}_3(x) = \prod_{n \geq 1}
\left(e^{x^2} \left(1 - \dfrac{x^2}{\left(n- \frac{1}{2}\right)^2}\right)^{\left(n-\frac{1}{2}\right)^2}\right).
$$
Note that the triple cosine function can also be written as
\begin{align*}
{\mathcal C}_3(x) 
&= \displaystyle 
\prod_{n \geq 1} \left(e^{x^2} \left(1-\dfrac{4x^2}{(2n-1)^2}
\right)^{\frac{(2n-1)^2}{4}}\right) \\
&= \displaystyle 
\prod_{n \geq 1} \left(e^{4 x^2} \left(1 - \dfrac{4x^2}{(2n- 1)^2}\right)^{(2n-1)^2}\right)^{\frac{1}{4}},
\end{align*}
which can be seen as a bisection (only taking the odd indices) of
the product $({\mathcal S}_3(x) e^{-\frac{x^2}{2}})^{1/4}$ with 
$x$ replaced by $2x$.

\medskip

Below we will exemplify generalizations of the Ramanujan product 
identity and the Kurokawa--Rovinski\u{i} triple trigonometric 
functions, as well as multisections of such products, by analogy 
with our generalization of Gosper's product evaluation. 
Furthermore, we will see how recent results of the third author
\cite{Holcombe2026} enter this framework.

\subsection{Multisections of Ramanujan's product}
\label{subsectionMotivating}

Given an infinite series or infinite product evaluation, a natural 
way of refining such an evaluation is given by the evaluation of the
corresponding multisections as we did previously. This motivates our
bisection of Ramanujan's product highlighted in 
Theorem~\ref{theoremmotive} below.

\begin{theorem}\label{theoremmotive}
The relation 
 \begin{multline*}
 \prod_{n\ge 0} e^{-x^2} \left( 1 + \frac{x^2}{(2n+1)^2} \right)^{(2n+1)^2} = 
 \\ 
 (1+e^{-\pi x})^{-x^2} \cdot
 \exp\left( \frac{2 x \operatorname{Li}_2\left(-e^{-\pi x}\right)}{\pi}+\frac{2 \operatorname{Li}_3\left(-e^{-\pi x}\right)}{\pi ^2}-\frac{\pi x^3}{6}+\frac{3 \zeta
 (3)}{2 \pi ^2} \right). 
 \end{multline*}
holds for complex $x$ such that $\Re(x) > 0$. 
\end{theorem} 

\begin{proof}

It is possible to prove this theorem by mimicking 
the proof for the evaluation of the Ramanujan's Equation~\eqref{Ramanujan} given in
\cite[p.~488]{Berndt1998}, but we will give a shorter proof 
using Ramanujan's evaluation. Note that a proof for a 
bisection is ``easier'' than for a multisection.

\medskip

\noindent
Define, as in \cite{AlloucheMarkov2026},
$$
{\mathcal S}^{\rm Ram}(x) := \
\prod_{n \geq 1} e^{-x^{2}} \left( 1 + \frac{x^2}{n^2} 
\right)^{n^{2}}.
$$
Then
$$
\left({\mathcal S}^{\rm Ram}\left(\frac{x}{2}\right)\right)^4 = \
\prod_{n \geq 1} e^{-x^{2}} \left( 1 + \frac{x^2}{4 n^2} 
\right)^{4 n^{2}} = \prod_{\substack{n \geq 1 \\ n \ \text{even}}}
e^{-x^{2}} \left( 1 + \frac{x^2}{n^2} \right)^{n^{2}}
$$
so that
$$
\frac{{\mathcal S}^{\rm Ram}(x)}
{\left({\mathcal S}^{\rm Ram}\left(\frac{x}{2}\right)\right)^4}
= \prod\limits_{\substack{n \geq 1 \\ n \ \text{odd}}}
e^{-x^{2}} \left( 1 + \frac{x^2}{n^2} \right)^{n^{2}} =
\prod_{n \geq 0} e^{-x^2} \ \left( 1 + \frac{x^2}{(2n+1)^2} \right)^{(2n+1)^2}.
$$
Now, consider the polylogarithm function $\operatorname{Li}_{s}(z) = \sum_{n\ge 1} \frac{z^n}{n^s}$ defined for $s$ a positive integer
and $|z| <1$ (the function $z \mapsto \operatorname{Li}_{s}(z)$ can be 
extended to the whole complex plane by analytic continuation). The 
polylogarithm satisfies the functional equation
\begin{equation}
\label{eqn:polylog}
\displaystyle \operatorname{Li} _{s}(-z) +
\operatorname{Li} _{s}(z)=2^{1-s}\operatorname {Li} _{s}(z^{2});
\end{equation}
see~\cite{Lewin1981} for a classical reference.
Using Equation~\eqref{Ramanujan} and the cases $s \in\{2, 3\}$ of Equation~\eqref{eqn:polylog}, we obtain that 
\begin{multline*}
 \prod\limits_{\substack{n \geq 1 \\ n \ \text{odd}}} e^{-x^2} \left( 1 + \frac{x^2}{(2n+1)^2} \right)^{(2n+1)^2} = 
 \\
 (1+e^{-\pi x})^{-x^2}
 \exp\left( \frac{2 x \operatorname{Li}_2\left(-e^{-\pi x}\right)}{\pi}+\frac{2 \operatorname{Li}_3\left(-e^{-\pi x}\right)}{\pi ^2}-\frac{\pi x^3}{6} + \frac{3 \zeta
 (3)}{2 \pi ^2} \right), 
\end{multline*}
 as desired. 
\end{proof}
 
\subsection{A cubic extension}\label{sec:cubic}

The same techniques allow us to evaluate 
$$ 
\prod _{n\ge 1} e^{-x^2} \left(1+\frac{x^2}{n^3}\right)^{n^3} 
$$
as the product of 
\begin{multline*}
 2 \sqrt{2} \pi ^{3/2} A^9 \Gamma^{x^2} \left(x^{2/3}+1\right) 
 \cdot \\ 
 \Gamma^{x^2} \left(1-\frac{i \left(-i+\sqrt{3}\right) x^{2/3}}{2} \right) \Gamma^{x^2}
 \left( 1 + \frac{i \left(i+\sqrt{3}\right) x^{2/3}}{2} \right) 
\end{multline*}
 and 

 \ 

\noindent $ \operatorname{exp}\Big( -6 \psi ^{(-4)}\left(x^{2/3}+1\right)-6 \psi ^{(-4)}\left(1-\sqrt[3]{-1} x^{2/3}\right) - 
 6 \psi ^{(-4)}\big((-1)^{2/3}
 x^{2/3} + 1\big) + $

\noindent $ \frac{3}{2} x^{2/3} \Big(4 \psi ^{(-3)}\left(x^{2/3}+1\right)+\left(-2-2 i \sqrt{3}\right) \psi ^{(-3)}\left(1-\sqrt[3]{-1}
 x^{2/3}\right) + 2 i \big(\sqrt{3}+ $
 
\noindent $ i\big) \psi ^{(-3)}\left((-1)^{2/3} x^{2/3}+1\right) + 
 x^{2/3} \Big(-2 \psi
 ^{(-2)}\big(x^{2/3} + 1\big) + 
 \left(1-i \sqrt{3}\right) \psi ^{(-2)}\big(1 - $
 
\noindent $ \sqrt[3]{-1} x^{2/3}\big) + 
 \big(1 + i \sqrt{3}\big) \psi
 ^{(-2)}\big((-1)^{2/3} x^{2/3} + 1\big)\Big)\Big)+\frac{9 \zeta (3)}{4 \pi ^2} \Big). $

\section{Further extensions of higher degree: the Wallis hierarchy}\label{sec:Wallis-hierarchy}

Looking at the example given in Section~\ref{sec:cubic}, one might want to
look at products of the form 
\begin{equation*}\label{higher-commongen}
 \prod_{n\ge 1} e^{r(n, x)} \left( 1 + \frac{x}{p_1(n)} \right)^{p_2(n)} 
\end{equation*}
as in Section~\ref{sec:higherorder}, but where the polynomials $p_1$ 
and $p_2$ do not necessarily have the same degree, and the quantity 
$r(n, x)$ is chosen as the ``minimal'' counterterm making the infinite 
product converge.

\bigskip

A particularly nice case is given in \cite{Holcombe2026}, where
the infinite product $G_m(x)$ is defined by
\[
G_m(x) := \prod_{n \geq 2} \exp\left(\sum_{1 \leq j \leq \lfloor\frac{m+1}{2}\rfloor} \frac{x^{2j}n^{m-2j}}{j}\right) 
\left(1 - \frac{x^2}{n^2}\right)^{n^m}
\]
and $Q_m := G_m(1)$ can be seen a generalized Wallis formula. 
Namely, the (inverse of the) classical Wallis formula reads
$$
\frac{2}{\pi} = \prod_{n \geq 1}\left(1 - \frac{1}{4n^2}\right).
$$
This product only involves the even numbers: the corresponding 
product with {\it all}\, the integers $\geq 2$ is 
$$
\prod_{n\ge 2}\left(1-\frac1{n^2}\right) = G_0(1).
$$
The third author has studied a natural generalization of this product,
that he has called the {\it Wallis hierarchy}: the first few values of 
these products are (see \cite{Holcombe2026}):
\[
\begin{array}{lll}
Q_0 &= \displaystyle
\prod_{n\ge 2}\left(1-\frac1{n^2}\right)
&= \ \ \dfrac12 \\
Q_1 &= \displaystyle
\prod_{n\ge 2} e^{\frac1n} \left(1-\frac1{n^2}\right)^n
&= \ \ \dfrac{e^{\EMcst}}{2} \\
Q_2 &= \displaystyle
\prod_{n\ge 2} e \left(1-\frac1{n^2}\right)^{n^2}
&= \ \dfrac{\pi}{e^{3/2}} \\
Q_3 &= \displaystyle
\prod_{n\ge 2} e^{n+\frac{1}{2n}}\left(1-\frac1{n^2}\right)^{n^3}
&= \ \ \displaystyle
\sqrt{\frac{\pi}{2}}\,
\exp\left(\EMcst-\frac76-\frac{3\zeta'(2)}{\pi^2}
\right) \\
Q_4 &= \displaystyle
\prod_{n\ge 2} e^{n^2+\frac12}\left(1-\frac1{n^2}\right)^{n^4}
&= \ \ \displaystyle 
\pi \exp\left(-\frac74 + \frac{3\zeta(3)}{\pi^2}\right).
\end{array}
\]

\begin{remark}
The value of $Q_0$ is easy to obtain by telescopic simplification
(write $1 - \frac{1}{n^2} = (\frac{n-1}{n})(\frac{n+1}{n})$).
Also note that the value of $Q_2$ was first given in 
\cite{Holcombe2013}.
\end{remark}

\bigskip

  To obtain formulas of this sort, we propose a slight modification of     the proof given in~\cite{Holcombe2026} by first taking the 
 logarithmic  derivative of the involved infinite products, as was done above. 
Namely, for \(|x|<2\),
\[
\begin{aligned}
\frac{d}{dx}\log G_m(x)
&=
\sum_{n\ge 2}
\left[
\sum_{j=1}^{\lfloor \frac{m+1}{2}\rfloor}2x^{2j-1}n^{m-2j}
-
\frac{2xn^m}{n^2-x^2}
\right] \\
&=
2\sum_{n\ge 2}
\left[
\sum_{j=1}^{\lfloor \frac{m+1}{2}\rfloor}x^{2j-1}n^{m-2j}
-
\sum_{j\ge 1} x^{2j-1}n^{m-2j}
\right] \\
&=
-2\sum_{j\ge \lfloor \frac{m+1}{2}\rfloor + 1} x^{2j-1}
\sum_{n\ge 2} n^{m-2j} \\
&=
-2\sum_{j\ge \lfloor \frac{m+1}{2}\rfloor + 1}
x^{2j-1}\bigl(\zeta(2j-m)-1\bigr).
\end{aligned}
\]
The counterterm has canceled exactly the non-summable 
part of the logarithmic expansion.
Since \(G_m(0)=1\), integration from \(0\) to \(x\) gives
\[
\log G_m(x)
=
-\sum_{j\ge \lfloor\frac{m+1}{2}\rfloor+1}
\frac{x^{2j}}{j}\bigl(\zeta(2j-m)-1\bigr).
\]
Putting \(x=1\) gives
\[
\log Q_m
=
\log G_m(1)
=
-\sum_{j\ge \lfloor\frac{m+1}{2}\rfloor+1}
\frac{\zeta(2j-m)-1}{j}.
\]

Now we distinguish two cases according to the parity of $m$.

\subsection*{Even branch}
 Let $m = 2a$, with \(a\geq 1\). Then \(\lfloor\frac{m+1}{2}\rfloor=a\), so that $$ \frac{d}{dx}\log G_{2a}(x) = -2x^{2a-1} \sum_{r \ge  
 1}\bigl(\zeta(2r)-1\bigr)x^{2r}. $$
Using (see, e.g., \cite[Eq.~18, p.~271]{SrivastavaChoi2012})
\[
\sum_{r\ge 1} \bigl(\zeta(2r)-1\bigr)x^{2r}
=
\frac{1-\pi x\cot(\pi x)}{2}
-
\frac{x^2}{1-x^2},
\]
one obtains
\[
\log Q_{2a}
=
-2\int_0^1
x^{2a-1}
\left[
\frac{1-\pi x\cot(\pi x)}{2}
-
\frac{x^2}{1-x^2}
\right]dx.
\]

 A routine application of integration by parts then gives that 
\begin{equation*}
\log Q_{2a}
=
\log\pi
-
H_a
-
\frac1{2a}
+
2a
\sum_{\ell=1}^{a-1}
(-1)^{\ell+1}
\frac{(2a-1)!}{(2a-2\ell)!}
\frac{\zeta(2\ell+1)}{(2\pi)^{2\ell}}
\end{equation*}
(where $H_a$ is the $a$-th harmonic number, 
$H_a = 1 + \frac{1}{2} + \ldots + \frac{1}{a}$).

\subsection*{Odd branch}

Let \(m=2a+1\), with \(a\geq0\). 
Then \(\lfloor\frac{m+1}{2}\rfloor=a+1\), so
\[
\frac{d}{dx}\log G_{2a+1}(x)
=
-2x^{2a+1}
\sum_{r\ge 1} \bigl(\zeta(2r+1)-1\bigr)x^{2r}.
\]
Using (see, e.g., \cite[Eq.~16, p.~271]{SrivastavaChoi2012}
\[
\sum_{r\ge 1} \bigl(\zeta(2r+1)-1\bigr)x^{2r}
=
-\frac12\left[\psi(1+x)+\psi(1-x)+2\gamma\right]
-
\frac{x^2}{1-x^2},
\]
one obtains
\[
\log Q_{2a+1}
=
\int_0^1
x^{2a+1}
\left[
\psi(1+x)+\psi(1-x)+2\gamma
+
\frac{2x^2}{1-x^2}
\right]dx.
\]
The apparent singularities at \(x=1\) cancel inside the bracket, so the integral is an ordinary convergent integral.

\bigskip

The rest of the proof is the same as in \cite{Holcombe2026}.

\bigskip

\begin{remark}
 Though the equalities giving $Q_1, \cdots, Q_4, \cdots$ are generalizations of the equality giving $Q_0$, one might find that the ``true'' 
 Wallis product is not $Q_0$, but rather $\displaystyle\frac{2}{\pi} = \prod_{n \geq 1}\left(1 - \frac{1}{4n^2}\right)$. Hence one might want 
 to consider generalizations of this product
where the quantity $(1 - \frac{1}{n^2})^{n^m}$ is replaced with
$(1 - \frac{1}{4n^2})^{n^m}$ and the exponential term adequately 
adjusted. This can indeed be done by looking at $G_m(\frac{1}{2})$
instead of $G_m(1)$. We leave this to the reader.
\end{remark}

\section{A Dirichlet series-based approach}
\label{sec:Dirichlet-approach}

 Another approach to proving Gosper's equality (and related formulas) is to use Dirichlet series. Below, we develop this method 
 for Gosper's equality.
Let $P$ be Gosper's product as in Equation~\eqref{Gosper}.
 Since $$ -1+\left(3n+\frac12\right)\log\left(1+\frac{1}{3n}\right) = 
 \frac{1}{108n^2}-\frac{1}{324n^3}+O(n^{-4}), $$ the infinite product shown in Equation~\eqref{Gosper} converges, with 
	\[
	\log P=\lim_{N\to +\infty}S_N,
	\qquad
	S_N:=\sum_{n=1}^N
	\left\{
	-1+\left(3n+\frac12\right)
	\log\left(1+\frac{1}{3n}\right)
	\right\}.
	\]
	For \(a>0\), put
	\[
	W_N(a):=\sum_{k=0}^{N-1}(k+a)\log(k+a),
	\qquad
	L_N(a):=\sum_{k=0}^{N-1}\log(k+a),
	\]
	and
	\[
	Z_N(s,a):=\sum_{k=0}^{N-1}(k+a)^{-s}.
	\]
	From
	\[
	Z_N(s,a)=\zeta(s,a)-\zeta(s,N+a)
	\]
	it follows that
	\[
	W_N(a)=-\zeta'(-1,a)+\zeta'(-1,N+a),
	\qquad
	L_N(a)=-\zeta'(0,a)+\zeta'(0,N+a).
	\]
	Euler--Maclaurin gives, as \(q\to +\infty\),
	\begin{align*}
		\zeta'(-1,q)
		&=
		\frac12q^2\log q-\frac14q^2-\frac12q\log q
		+\frac1{12}\log q+\frac1{12}
		+\frac{1}{720q^2}+O(q^{-4}), \\
		\zeta'(0,q)
		&=
		\left(q-\frac12\right)\log q-q
		+\frac{1}{12q}-\frac{1}{360q^3}+O(q^{-5}).
	\end{align*}
	Writing
	\[
	B_2(a):=a^2-a+\frac16
	\]
	and expanding at \(q=N+a\) gives
	\begin{align*}
		W_N(a)
		&=
		\frac12N^2\log N-\frac14N^2
		+\left(a-\frac12\right)N\log N
		+\frac12B_2(a)\log N \\
		&\qquad
		-\zeta'(-1,a)+\frac12B_2(a)
		+\frac{a(2a^2-3a+1)}{12N}
		+O(N^{-2}), \\
		L_N(a)
		&=
		N\log N-N+\left(a-\frac12\right)\log N
		-\zeta'(0,a)+\frac{B_2(a)}{2N}
		+O(N^{-2}).
	\end{align*}
	Now
	\[
	\log\left(1+\frac{1}{3n}\right)
	=\log\left(n+\frac13\right)-\log n,
	\qquad
	3n+\frac12=3\left(n+\frac13\right)-\frac12.
	\]
	After replacing \(n\) by \(k+1\),
	\[
	S_N
	=-N+3W_N\left(\frac43\right)-3W_N(1)
	-\frac12L_N\left(\frac43\right)-\frac12L_N(1).
	\]
	Since \(B_2(4/3)=11/18\) and \(B_2(1)=1/6\), the preceding
	expansions give
	\begin{align*}
		3\left(W_N\left(\frac43\right)-W_N(1)\right)
		&=
		N\log N+\frac23\log N
		-3\zeta'\left(-1,\frac43\right)+3\zeta'(-1) \\
		&\qquad
		+\frac23+\frac{5}{27N}+O(N^{-2}), \\
		-\frac12\left(L_N\left(\frac43\right)+L_N(1)\right)
		&=
		-N\log N+N-\frac23\log N \\
		&\qquad
		+\frac12\zeta'\left(0,\frac43\right)
		+\frac12\zeta'(0)
		-\frac{7}{36N}+O(N^{-2}).
	\end{align*}
	Hence
	\[
	\begin{aligned}
		S_N
		&=
		-3\zeta'\left(-1,\frac43\right)+3\zeta'(-1)
		+\frac12\zeta'\left(0,\frac43\right)
		+\frac12\zeta'(0) \\
		&\qquad
		+\frac23-\frac{1}{108N}+O(N^{-2}),
	\end{aligned}
	\]
	and therefore
	\[
	\log P
	=-3\zeta'\left(-1,\frac43\right)+3\zeta'(-1)
	+\frac12\zeta'\left(0,\frac43\right)
	+\frac12\zeta'(0)+\frac23.
	\]
	Differentiating the shift relation $\zeta(s,a+1)=\zeta(s,a)-a^{-s}$ with 
	with respect to \(s\) gives $	\zeta'(s,a+1)=\zeta'(s,a)+a^{-s}\log a$.
	In particular,
	\[
	\zeta'\left(-1,\frac43\right)
	=\zeta'\left(-1,\frac13\right)-\frac13\log 3,
	\qquad
	\zeta'\left(0,\frac43\right)
	=\zeta'\left(0,\frac13\right)-\log 3.
	\]
	Thus
	\begin{equation}
    \label{eq:new-eqn-P}
		\log P=
		-3\zeta'\left(-1,\frac13\right)
		+\frac12\zeta'\left(0,\frac13\right)
		+3\zeta'(-1)+\frac12\zeta'(0) 
		+\frac12\log 3+\frac23.
		\end{equation}
Now, we have
\begin{align*}
\zeta(0) &= -1/2, \\
\zeta'(0) &= - \frac{1}{2} \log 2\pi, \\
\zeta'\left(-1,\frac13\right)
		&=
		\frac{\psi'(1/3)}{12\pi\sqrt3}
		-\frac13\zeta'(-1)
		-\frac{\pi}{18\sqrt3}
		-\frac{\log 3}{72}, \\
\zeta'\left(0,\frac13\right)
		&= \log\Gamma\left(\frac13\right)-\frac12\log(2\pi), \\
\zeta'(-1) &= \frac{1}{12} - \log A, 
\end{align*}
where the first two equalities are classical, where, for the third 
equality, we used \cite[Eq.~(18) in Proposition~3]{MillerAdamchik1998} 
or \cite[Section~1.2]{Brychkov2008}, and where $A$ is the Glaisher--Kinkelin constant defined in Section~\ref{definitions-lemmas}. 
Substituting these values in Equation~\eqref{eq:new-eqn-P} yields
	\[
	\log P
	=1+\frac12\log\left(\frac{\Gamma(1/3)}{2\pi}\right)
	+\frac{13}{24}\log 3-4\log A
	+\frac{2\pi^2-3\psi'(1/3)}{12\pi\sqrt3}.
	\]
	Consequently,
	\[
	\prod_{n\geq 1}\frac{1}{e}
	\left(1+\frac{1}{3n}\right)^{3n+\frac12}
	=
	\sqrt{\frac{\Gamma(\frac{1}{3})}{2\pi}} \ \frac{3^{\frac{13}{24}} 
 \ e^{1+\frac{2\pi^2 - 3\psi\,'(\frac{1}{3})}{12\pi\sqrt{3}}} }{A^4},
	\]
as desired.

\bigskip

To close this section about infinite products and Dirichlet series,
we would like to cite the paper \cite{Perkins-Van-Gorder} that studies
nice infinite products involving number theoretic functions, and where indices are either positive integers or prime numbers.

\section{Conclusion} While the infinite products studied in this paper are convergent products, one might want to study similar 
 ``zeta-regularized products'': such products have been alluded to in the introduction (where we have pointed to 
 \url{https://math.stackexchange.com/questions/4343366}). 
In particular, it might be fruitful to study possible links between
the present paper and a recent preprint of Hu and Kim \cite{HuKim2026}.


\subsection*{Acknowledgements}
Manon Stipulanti is an FNRS Research Associate supported by the Research grant 1.C.104.24F.

\subsection*{Note added on October 2, 2026}

We have found on the arXiv two preprints that are certainly
worth citing, namely J.~W.~E. Farrell's papers
\url{https://arxiv.org/abs/1906.00122} and
\url{https://arxiv.org/abs/2608.29741}.


\bibliographystyle{plainurl}
\bibliography{biblio.bib}

\end{document}